\documentclass[11pt,reqno]{amsart}
\usepackage{amsmath,amsthm,amssymb,natbib,geometry,bm,indentfirst,lineno,setspace,fullpage,enumitem,mathrsfs}
\usepackage[colorlinks,linkcolor=blue]{hyperref}
\usepackage[T1]{fontenc}
\numberwithin{equation}{section}

\title{Asymptotic Enumeration of $F$-decompositions in Typical Hypergraphs and of Universal Cycles}

\author{Danni Li}
\address{School of Mathematics and Statistics,
Beijing Institute of Technology, China}
\email{lidanni\_dany@163.com}

\author{Bin Wang}
\address{School of Mathematics and Statistics,
Beijing Institute of Technology, China}
\email{bin.wang@bit.edu.cn}

\author{Junxue Zhang}
\address{School of Science,
Hebei University of Technology, China}
\email{jxuezhang@163.com}
\date{}


\newtheorem{thm}{Theorem}[section]
\newtheorem{dfn}[thm]{Definition}
\newtheorem{lem}[thm]{Lemma}
\newtheorem{prop}[thm]{Proposition}

\newtheorem{claim}{Claim}

\begin{document}
\begin{abstract}
    In this paper, we determine the asymptotic number of $F$-decompositions in typical hypergraphs. More precisely, let $F$ be a non-trivial $f$-vertex $k$-graph. We show that, for all sufficiently large $n$, every $n$-vertex $F$-divisible $(n^{-\zeta},h,p)$-typical $k$-graph $G$ has $n^{\frac{f-k+o_n(1)}{|F|}|G|}$ $F$-decompositions.

    As an application, we asymptotically enumerate universal cycles for $k$-subsets of $[n]$, or equivalently, tight Euler tours of the complete $k$-graph $K_n^k$. For every $k\ge2$ and all sufficiently large $n$ satisfying $k\mid {n-1 \choose k-1}$, the number of such universal cycles is $n^{(1-\gamma(n)){{n \choose k}}}$, where $0<\gamma(n)=o_n(1)$. More generally, for $\ell\in[k-1]$, we show that a sufficiently typical $n$-vertex $k$-graph $G$ satisfying the corresponding divisibility condition has $n^{(k-\ell-\gamma'(n))|G|}$ Euler $\ell$-tours, where $0<\gamma'(n)=o_n(1)$.
\end{abstract}

\maketitle
\section{Introduction}
\subsection{$F$-decompositions and $F$-designs}
Let $G$ and $F$ be $k$-uniform hypergraphs ($k$-graphs for short). An $F$-\emph{decomposition} of $G$ is a collection of copies of $F$ in $G$ such that every hyperedge of $G$ is contained in exactly one of these copies. More generally, an $(F,\lambda)$-\emph{design} of $G$ is a collection of copies of $F$ in $G$ such that each hyperedge of $G$ is contained in exactly $\lambda$ copies. Consequently, an $F$-decomposition of $G$ is precisely an $(F,1)$-design of $G$. When $F$ and $G$ are complete hypergraphs, such a design is commonly referred to as a block design, and when additionally $\lambda=1$, it is known as a Steiner system.

The existence conjecture for block designs is a classical conjecture in combinatorics, and the corresponding problem for $(F,\lambda)$-designs of $G$ is a natural generalization. A $k$-graph $G$ is said to be $(F,\lambda)$-\emph{divisible} if $d_F(i)\mid \lambda d_G(S)$ for every $0\le i\le k-1$ and $S\in {V(G)\choose i}$, where $d_F(i):=\gcd\{d_F(S'):S'\in {V(F)\choose i}\} \text{ for }1\le i\le k-1$, and $d_F(0):=|F|$. This condition is necessary for the existence of an $(F,\lambda)$-design of $G$, while the existence conjecture asks whether this necessary divisibility condition is also sufficient. Specifically, the existence conjecture for block designs states, for fixed \(f,k,\lambda\) and for sufficiently large \(n\), a \(\left(K_f^k,\lambda\right)\)-design of \(K_n^k\) exists whenever \(K_n^k\) is \(\left(K_f^k,\lambda\right)\)-divisible. In 1847, Kirkman \cite{Kirkman1847} proved that a Steiner triple system of order $n$ (i.e. a $K_3$-decomposition of the complete graph $K_n$) exists if and only if $n\equiv 1,3\pmod{6}$. Wilson \cite{MR304203,MR304204,MR366695} resolved the classical existence conjecture of block designs for $k=2$. Moreover, Wilson \cite{MR396347} proved the stronger result that, for every fixed graph $F$, every sufficiently large $F$-divisible complete graph admits an $F$-decomposition. Keevash \cite{keevash2024existencedesigns,keevash2018existencedesignsii,keevash2018hypergraphmatchingsdesigns} subsequently resolved the classical existence conjecture of block designs for arbitrary $k$, and also obtained more general results in pseudorandom settings. Later, Glock, K\"{u}hn, Lo and Osthus \cite{glock2020existencedesignsiterativeabsorption} resolved the existence conjecture of an $(F,\lambda)$-design of $G$, where $F$ is an arbitrary fixed $k$-graph and $G$ is a sufficiently large typical $k$-graph, as well as in certain more general settings. More recently, alternative proofs and refinements have been obtained in \cite{delcourt2024proofhighgirthexistence,delcourt2024refinedabsorptionnewproof,keevash2024short,postle2025refinedabsorptionnewproof}. These existence results naturally lead to the corresponding enumeration problem.

A landmark result is that Keevash \cite{2015Counting} obtained an asymptotic count for \(\left(K_f^k,\lambda\right)\)-designs of \(K_n^k\). For an arbitrary fixed $k$-graph \(F\) and a sufficiently large typical $k$-graph $G$, Glock, K\"{u}hn, Lo and Osthus \cite{glock2020existencedesignsiterativeabsorption} observed that their existence theorem, combined with the approach of \cite{2015Counting}, may also yield an asymptotic count of $(F,\lambda)$-designs of $G$, but they did not give an explicit counting theorem. Keevash \cite{keevash2018existencedesignsii} also observed that his methods yield more general counting results. Building on these ideas, and inspired by a simplification introduced by Kwan \cite{kwan2020almost} in the study of random Steiner triple systems, we establish an asymptotic enumeration theorem for \(F\)-decompositions in typical hypergraphs. Its applicability to arbitrary fixed \(F\) is particularly useful for our enumeration of universal cycles and Euler tours.

\subsection{Universal cycles and Euler tours}

Universal cycles were introduced by Chung, Diaconis and Graham \cite{1992Universal} as generalizations of de Bruijn cycles \cite{debruijn1946}. Let $n>k>1$ be integers. A \textit{universal cycle} for $k$-subsets of $[n]$ is a cyclic sequence $\mathcal{C}_{n,k}=\left(a_1,a_2,\ldots,a_{\binom {n}{k}}\right)$ with $a_i\in[n]$ such that, as $i$ ranges over $\left[\binom{n}{k}\right]$, the sets $\{a_i,a_{i+1},\ldots,a_{i+k-1}\},$ where the indices are taken modulo $\binom nk$, are precisely the $k$-subsets of $[n]$, each occurring exactly once. Two universal cycles are regarded as the same if one can be obtained from the other by a cyclic shift or by reversing the cyclic order.

A necessary condition for the existence of such a universal cycle is $k\mid {n-1\choose k-1}.$ Indeed, each element of $[n]$ belongs to $\binom{n-1}{k-1}$ distinct $k$-subsets, while each occurrence of an element in the cyclic sequence belongs to exactly $k$ consecutive $k$-sets. Hence every element of $[n]$ must occur exactly $\frac{1}{k}{n-1\choose k-1}$ times in $\mathcal C_{n,k}$, which yields the necessary condition. Chung, Diaconis and Graham \cite{1992Universal} conjectured that, for every fixed $k$, this necessary condition is also sufficient for all sufficiently large $n$. Several partial results were obtained in \cite{2000On,MR1226137,MR3066357}, and the conjecture was eventually resolved by Glock, Joos, K\"{u}hn and Osthus \cite{2020Euler}. Their proof is based on a connection between universal cycles and tight Euler tours in $K_n^k$. A \emph{tight Euler tour} of $K_n^k$ is defined as a cyclic sequence of vertices of $K_n^k$ in which every hyperedge of $K_n^k$ appears exactly once as $k$ consecutive elements. By identifying each vertex of $K_n^k$ with an element of $[n]$ and each hyperedge with a $k$-subset of $[n]$, a universal cycle $\mathcal{C}_{n,k}$ corresponds exactly to a tight Euler tour of $K_n^k$. For further generalizations of Euler tours in hypergraphs, we refer the reader to \cite{MR5052126, MR3691547,MR2745697, MR3711683, MR3843267, MR4622724}.

The enumeration of universal cycles has been explicitly highlighted as an important problem in the area; see, for example, Problem 477 in \cite{Brad2009Research} and the subsequent work \cite{2022Enumerations}. In this paper, we asymptotically enumerate universal cycles for \(k\)-subsets of \([n]\), as an application of Theorem \ref{Thm:count_decom}. Moreover, the correspondence between universal cycles and tight Euler tours motivates us to consider more general Euler tours in hypergraphs. Specifically, we consider Euler $\ell$-tours with $0<\ell<k$. An \emph{Euler $\ell$-tour} of a $k$-graph $G$ is defined as a cyclic vertex sequence (with repetition) in which each edge of $G$ appears exactly once as $k$ consecutive vertices, and consecutive edges overlap in exactly $\ell$ consecutive vertices in the induced order. Note that the case $\ell=k-1$ corresponds exactly to a tight Euler tour. Actually, Glock, Joos, K\"{u}hn and Osthus \cite{2020Euler} proved a stronger result that typical hypergraphs admit tight Euler tours. This naturally leads to the existence and enumeration problems for Euler $\ell$-tours with general $\ell$ in typical hypergraphs. In this paper, we resolve both problems.

\subsection{Our results.}
Our first main result is the following asymptotic enumeration theorem for \(F\)-decompositions in typical hypergraphs.
\begin{thm}\label{Thm:count_decom}
    Let $\zeta,p\in \mathbb{R}^+$ be fixed with $0<p\le 1$. Let $k,f,h\in \mathbb{N}^+$ be fixed with $1<k< f$ and $h\ge 2^k{2f\cdot f!+k \choose k}$. Let $n\in \mathbb{N}^+$ be sufficiently large. Suppose $F$ is an $f$-vertex $k$-graph with $|F|\ge 1$ and $G$ is an $n$-vertex $F$-divisible $k$-graph which is $\left(n^{-\zeta},h,p\right)$-typical. Then the number of $F$-decompositions of $G$ is $n^{\frac{f-k+o_n(1)}{|F|}|G|}$.
\end{thm}

To establish the lower bound, we analyze a random $F$-removal process in a typical host hypergraph. Random removal processes and related random greedy methods have been extensively studied; see, for example, \cite{MR4015657,MR3350225,MR1445383,joos2025hypergraphremovalprocess,keevash2026terminalleaverandomgreedy}. The upper bound follows from a straightforward greedy counting argument.

As applications of Theorem \ref{Thm:count_decom}, we obtain the following asymptotic enumeration of universal cycles for $k$-subsets of $[n]$ and of Euler $\ell$-tours with general $\ell$ in typical hypergraphs.

\begin{thm}\label{Thm:uc}
    Let $k\in\mathbb{N}^+$ be fixed with $k>1$ and $n\in\mathbb{N}^+$ be sufficiently large satisfying $k\mid {n-1 \choose k-1}$. Then the number of universal cycles for \(k\)-subsets of \([n]\) is $n^{(1-\gamma(n)){{n \choose k}}}$, where $0<\gamma(n)=o_n(1)$.
\end{thm}

\begin{thm}\label{Thm:lcyclecounting}
    Let $k,\ell\in \mathbb{N}^+$ be fixed with $0<\ell <k$ and $\zeta, p\in \mathbb{R}^+$ be fixed with $0<p\le 1$. Let $n\in \mathbb{N}^+$ be sufficiently large. Then the following holds for all sufficiently large $h=h(n)\in\mathbb{N}^+$. Let $G$ be an $\left(n^{-\zeta},h,p\right)$-typical $k$-graph on $n$ vertices. Suppose $\frac{k}{k-\ell}\mid d_G(v)$ for every $v\in V(G)$ when $(k-\ell)\mid k$, and no additional divisibility condition when $(k-\ell)\nmid k.$ Then
    \begin{itemize}
        \item[$(1)$] the $k$-graph $G$ has an Euler $\ell$-tour;
        \item[$(2)$] the number of distinct Euler $\ell$-tours of $G$ is $n^{(k-\ell-\gamma(n))|G|}$, where $0<\gamma(n)=o_n(1)$.
    \end{itemize}
\end{thm}

The rest of this paper is organized as follows. In Section~\ref{Sec:Definitions and Preliminaries}, we introduce the necessary notation, definitions, and preliminaries. In Section~\ref{sec:main thm}, we prove the main theorem on the counting of $F$-decompositions. In Section~\ref{Sec:UC}, we apply Theorem \ref{Thm:count_decom} to count universal cycles for $k$-subsets of $[n]$ and Euler $\ell$-tours in typical hypergraphs. Finally, in Section~\ref{Sec:Conclusion}, we discuss open problems and conclude the paper.

\section{Notation and preliminaries}\label{Sec:Definitions and Preliminaries}
\subsection{Notation and definitions}
For $n\in \mathbb{N}^+$, let $[n]$ denote the set $\{1,2,\ldots,n\}$ and ${[n]\choose k} $ denote all $k$-subsets of $[n]$. Let $n!$ denote the factorial of $n$ with the convention $0!=1$. For integers $n$ and $k \ge 1$, the falling factorial is defined by $(n)_{k} := \prod_{i=0}^{k-1} (n-i)$. Write $x\wedge y=\min\{x,y\}$. Let $f(n)$ and $g(n)$ be positive functions, and we write $f(n) = o_n(g(n))$ if and only if $ \lim_{n\to\infty} \frac{f(n)}{g(n)} = 0$; $f(n) = O_n(g(n))$ if and only if there exist a constant $c > 0$ and $n_0 \in \mathbb{N}$ such that $0 \leq f(n) \leq cg(n) \text{ for all } n \geq n_0$; $f(n) = \Omega_n(g(n))$ if and only if there exist a constant $c > 0$ and $n_0 \in \mathbb{N}$ such that $f(n) \ge cg(n) \text{ for all } n \geq n_0$. We may omit the subscript ``$n$'' of the notation when the context is clear. We say that ``$n$ is sufficiently large'' if there exists $n_0\in\mathbb{N}^+$ such that the property under consideration holds for all $n\ge n_0$. For $a,b\in\mathbb{R}^+$, we write $a\ll b$ to mean that $a$ is sufficiently small relative to $b$; that is, there exists $a_0=a_0(b)\in\mathbb{R}^+$ such that the property under consideration holds whenever $0<a\le a_0$. Note that both $n_0$ and $a_0$ may also depend on any previously fixed parameters.

Let $G$ be a $k$-uniform hypergraph, let $V(G)$ denote the vertex set of $G$ and $E(G)$ denote the hyperedge set of $G$. For simplicity, we just use $|G|$ to denote the number of hyperedges of $G$. For $S\subseteq V(G)$, let $N_G(S)$ denote the set of neighbours of $S$ in $G$, i.e. $N_G(S) = \{e\setminus S \mid e \in E(G), S \subseteq e \}$, and $d_G(S)$ denote the degree of $S$ in $G$, i.e. $d_G(S)=\left|N_G(S)\right|$. For $1\le i\le k-1,$ let $\Delta_{i}(G)$ denote the maximum $i$-degree of $G$, i.e. $\Delta_i(G):=\max\{d_G(S):S\subseteq V(G), |S|=i\}$. Suppose $H$ is a subhypergraph of $G$. Let $G\setminus H$ denote the subhypergraph with vertex set $V(G)$ and edge set $E(G)\setminus E(H)$. Suppose $F$ is a $k$-uniform hypergraph. An $F$-copy in $G$ is the image of an embedding of $F$ into $G$, including the images of all isolated vertices of $F$. Let $\mathcal{F}(G)$ denote the collection of all copies of $F$ in $G$. For every edge $e\in E(G)$, $\mathcal{F}_e(G)$ denotes the collection of copies of $F$ containing $e$ in $G$. For $0\le p\le 1,$ the $p$-random subhypergraph $G_p$ of $G$ is obtained by independently including each hyperedge of $G$ with probability $p$. We say that an event occurs ``with high probability'' (w.h.p. for short) if its probability tends to 1 as some parameter grows to infinity. Later we will simply refer to a $k$-uniform hypergraph as a $k$-graph, a hyperedge as an edge and a subhypergraph as a subgraph. Throughout the paper, $F$ and $G$ are assumed to be non-trivial, i.e., $|F|\ge1$ and $|G|\ge1$.

Typical $k$-graphs are an important class of hypergraphs studied in this paper. Informally, typicality describes how closely the sizes of common neighbourhoods in a $k$-graph resemble those expected in a random $k$-graph.

\begin{dfn}[$(c,h,p)$-typical]
    Let $c,p\in\mathbb{R}$ and $k,h\in \mathbb{N}^+$ with $0\le c,p\le 1<k$. A $k$-graph $G$ is called $\left(c,h,p\right)$-\textbf{typical} if for every $A\subseteq {V(G) \choose k-1}$ with $|A|\le h$, we have $\left|\underset{S\in A}{\bigcap}N_G(S)\right|=(1\pm c) p^{|A|}|V(G)|$.
\end{dfn}

Then we formally give the definition of packings and decompositions.
\begin{dfn}[Packings and decompositions]
    Let $k,m\in \mathbb{N}^+$ with $k>1$. Let $F$ and $G$ be $k$-graphs. An $F$-\textbf{packing} $\mathcal{F}$ of $G$ is a collection $\{F_1,\dots,F_m\}$ of edge-disjoint subgraphs of $G$ where each $F_i$, $ i\in[m]$, is isomorphic to $F$. An $F$-\textbf{decomposition} is an $F$-packing that covers every edge of $G$ exactly once.
\end{dfn}
The divisible property is necessary for the existence of decompositions of hypergraphs.
\begin{dfn}[$F$-divisible]
    Let $k\in \mathbb{N}^+$ with $k>1$. Let $G$ and $F$ be $k$-graphs. For $ 1\le i\le k-1$, define $d_F(i):=\gcd\left(d_F(S):S\in {V(F)\choose i}\right)$ and $d_F(0):=|F|$. We say that $G$ is $F$-\textbf{divisible} if $d_F(i)\mid d_G(S)$ for every $0\le i\le k-1$ and $S\in {V(G)\choose i}$.

\end{dfn}

\subsection{Preliminaries}

We include the following definition of supermartingales, as well as Freedman's inequality to bound their tail probabilities.

\begin{dfn}[Supermartingale]
    Given a probability space $\Omega$ and filtration $\mathcal{F}=(\mathcal{F}_i)_{i\ge 0}$, a sequence $A=(A_i)_{i\ge 0}$ of random variables on $\Omega$ is a \textbf{supermartingale} for $\mathcal{F}$ if each $A_i$ is $\mathcal{F}_i$-measurable, $\mathbb E|A_i|<\infty$ and $\mathbb{E}[A_{i+1}\mid \mathcal{F}_i]\le A_i \text{ for every } i\ge 0$.
\end{dfn}

\begin{lem}[Freedman's inequality, \cite{freedman1975tail}]\label{Lem:Freedman}
    Suppose $a,b,v\in\mathbb{R}^+$. Let $(A_i)_{i\ge 0}$ be a supermartingale for $\mathcal{F}=\left(\mathcal{F}_i \right)_{i\ge 0} $ with each $|A_{i+1}-A_i|\le b$ for every $i\ge 0$. Let $E$ be the ``bad'' event that there exists $j>0$ with $A_j\ge A_0+a$ and $\sum_{i<j}\operatorname{Var}[A_{i+1}\mid \mathcal{F}_i]\le v$. Then $\mathbb{P}[E]\le e^{-\frac{a^2}{2(v+ba)}}$.
\end{lem}

Next we give a property of typical $k$-graphs. Let \(G \) and $H$ be \( k \)-graphs. An \emph{embedding} of $H$ into $G$ is an injective map $\varphi:V(H)\to V(G) $ such that $\varphi(e)\in E(G)$ for every $e\in E(H)$, where for every $S\subseteq V(H)$, we write $\varphi(S):=\{\varphi(v):v\in S\}.$ For $V_0\subseteq V(H)$, let $\psi:V_0\to V(G)$ be an embedding of $H[V_0]$ into $G$. We say that an embedding $\varphi:V(H)\to V(G)$ \emph{extends} $\psi$ if $\varphi(v)=\psi(v)$ for every $v\in V_0$. Let \( F \) be a \( k \)-graph. An \emph{automorphism} of $F$ is a bijection $\sigma:V(F)\to V(F)$ such that $\sigma(e)\in E(F) $ if and only if $e\in E(F)$ for every $e\in\binom{V(F)}{k}$. The set of all automorphisms of $F$ is denoted by $\operatorname{Aut}(F)$. We write $\operatorname{aut}_F:=|\operatorname{Aut}(F)|$ for the number of automorphisms of $F$.

\begin{prop}
    \label{prop:typical-count}
    Let $k,f,h,n\in\mathbb N^+$ and $p,c\in\mathbb R^+$ with $0<p\le 1<k<f$. Let $F$ be an $f$-vertex $k$-graph and let $G$ be a $(c,h,p)$-typical $n$-vertex $k$-graph. Suppose $h\ge\Delta_1(F)$ and $\frac{f-1}{p^hn}+c<1.$ Let $V_0\subseteq V(F)$ and let $\psi:V_0\to V(G)$ be an embedding of $F[V_0]$ into $G$. Then the number of embeddings of $F$ into $G$ extending $\psi$ is at least $\left(1-c-\frac{f-1}{p^hn} \right)^{f-|V_0|} p^{\left|F\setminus F[V_0]\right|}n^{f-|V_0|}$ and at most $(1+c)^{f-|V_0|} p^{\left|F\setminus F[V_0]\right|}n^{f-|V_0|}.$ In particular,
    \begin{itemize}
        \item[$(1)$] $\frac{1}{\operatorname{aut}_F} \left(1-c-\frac{f-1}{p^hn} \right)^f p^{|F|}n^f \le |\mathcal F(G)| \le \frac{1}{\operatorname{aut}_F} (1+c)^fp^{|F|}n^f;$
        \item[$(2)$]
        $\frac{|F|k!}{\operatorname{aut}_F} \left(1-c-\frac{f-1}{p^hn} \right)^{f-k} p^{|F|-1}n^{f-k} \le |\mathcal F_e(G)| \le \frac{|F|k!}{\operatorname{aut}_F} (1+c)^{f-k}p^{|F|-1}n^{f-k},$ for every edge $e\in E(G)$.
    \end{itemize}
\end{prop}

\begin{proof}
    Write $V(F)\setminus V_0=\{w_1,\ldots,w_b\},$ where $b=f-|V_0|$, and embed $w_1,\ldots,w_b$ into $V(G)$ successively. For $i\in[b]$, suppose that $w_1,\ldots,w_{i-1}$ have already been embedded, and let
    \[S_i:= \left\{ e\in E(F): w_i\in e,\ e\setminus\{w_i\} \subseteq V_0\cup\{w_1,\ldots,w_{i-1}\} \right\},\]
    and $s_i:=|S_i|$. Then $s_i\le d_F(w_i)\le\Delta_1(F)\le h.$ The images of the elements in $\{e\setminus\{w_i\}:e\in S_i\}$ form $s_i$ distinct $(k-1)$-subsets of $V(G)$. Hence, by the $(c,h,p)$-typicality of $G$, before excluding previously used vertices, there are $(1\pm c)p^{s_i}n$ choices for the image of $w_i$. Since at most $f-1$ vertices have already been used, the number of choices is at least $(1-c)p^{s_i}n-(f-1) \ge p^{s_i}n \left(1-c-\frac{f-1}{p^hn}\right)$ and at most $(1+c)p^{s_i}n.$ The sets $S_i$, $i\in[b]$, are pairwise disjoint, and $\sum_{i=1}^b s_i=\left|F\setminus F[V_0]\right|,$ since every edge of $F\setminus F[V_0]$ is counted exactly once, namely when its last vertex outside $V_0$ is embedded. Therefore, the number of embeddings of $F$ into $G$ extending $\psi$ is at least $\left(1-c-\frac{f-1}{p^hn} \right)^{b} p^{\left|F\setminus F[V_0]\right|}n^{b}$ and at most $(1+c)^{b} p^{\left|F\setminus F[V_0]\right|}n^{b}.$

    Taking $V_0=\emptyset$ and dividing the number of embeddings by $\operatorname{aut}_F$ gives
    \[\frac{1}{\operatorname{aut}_F} \left(1-c-\frac{f-1}{p^hn} \right)^f p^{|F|}n^f \le |\mathcal F(G)| \le \frac{1}{\operatorname{aut}_F} (1+c)^fp^{|F|}n^f.\]
    For a fixed $e\in E(G)$, choose an edge $e_F\in E(F)$ and a bijection $\psi:e_F\to e.$ Applying the first statement with $V_0=e_F$, we have the number of embeddings $\varphi:V(F)\to V(G)$ of $F$ into $G$ extending $\psi$ is at least $\left(1-c-\frac{f-1}{p^hn} \right)^{f-k} p^{|F|-1}n^{f-k}$ and at most $(1+c)^{f-k}p^{|F|-1}n^{f-k}.$ There are $|F|$ choices for $e_F$ and $k!$ choices for the bijection $\psi:e_F\to e$ and every copy $F\in\mathcal F_e(G)$ is counted exactly $\operatorname{aut}_F$ times. Hence
    \[\frac{|F|k!}{\operatorname{aut}_F} \left(1-c-\frac{f-1}{p^hn} \right)^{f-k} p^{|F|-1}n^{f-k} \le |\mathcal F_e(G)| \le \frac{|F|k!}{\operatorname{aut}_F} (1+c)^{f-k}p^{|F|-1}n^{f-k}.\]
    This completes the proof.
\end{proof}

\section{Counting $F$-decompositions in typical hypergraphs}\label{sec:main thm}
\subsection{Proof of Theorem \ref{Thm:count_decom}}
We will prove the lower bound and the upper bound of Theorem \ref{Thm:count_decom} separately. We first give some lemmas which will be used to prove the lower bound. We formally give the definition of random removal process as follows.

\begin{dfn}[Random removal process]\label{random process}
    Let $k\in\mathbb{N}^+$ with $k>1$. Let $G$ and $F$ be $k$-graphs, and let $\mathcal{F}\subseteq \mathcal{F}(G)$ be a subfamily of all copies of $F$ in $G$. The \textbf{random removal process} of $G$ with respect to $\mathcal{F}$ is defined as follows. Let $ {G}_0=G$ and $ {\mathcal{F}_0}=\mathcal{F}$. At each step $i\ge 1$, choose a copy of $F$ in $ {\mathcal{F}}_{i-1}$ uniformly at random, denote it by $ {F_{i-1}}$, and then define $ {G}_i= {G}_{i-1}\setminus {F}_{i-1}$ and define $ {\mathcal{F}}_{i}=\{F\in {\mathcal{F}}_{i-1}: E(F)\cap E( {F}_{i-1})=\emptyset\}$, i.e. $ {\mathcal{F}}_{i}=\mathcal{F}\cap \mathcal{F}(G_i)$. The termination time of the process is $M:=\min\{i\geq 0:\mathcal{F}_{i}=\varnothing\}$. Note that the process terminates after finitely many steps and $M\leq\lfloor\frac{|G|}{|F|}\rfloor$. At termination, the selected copies $F_{0},\ldots,F_{M-1}$ are pairwise edge-disjoint and form a maximal edge-disjoint subfamily of $\mathcal{F}$.
\end{dfn}

The following lemma is a crucial property of the random removal process, which will be proved in the next subsection.
\begin{lem}[Random removal lemma: typicality]\label{lem:Random_removal2}
    Let $ k, f,h \in \mathbb{N}^+$ be fixed with $1<k< f$. Let $\varepsilon,p,\zeta,\zeta'\in\mathbb{R}^+$ be fixed with $0<\zeta',\varepsilon\ll \zeta$ and $0<p\le1$. Let $n \in \mathbb{N}^+$ be sufficiently large. Let $F$ be an $f$-vertex $k$-graph. Let $G$ be an $n$-vertex $k$-graph which is $\left(n^{-\zeta},h,p\right)$-typical. Suppose $h\ge \Delta_1(F)$. Consider the random removal process of $G$ with respect to $\mathcal{F}(G)$ as in Definition \ref{random process}. Then with probability at least $1 - e^{-n^{\Omega(1)}},$ the process will not terminate before step $m: = \left\lfloor\frac{|G|}{|F|} (1 - n^{-\varepsilon})\right\rfloor$, and for every $i\in\{0,\dots,m\}$, $G_i$ is $\left(n^{-\zeta'},h,pp(i)\right)$-typical, where $p(i)=1 - \frac{|F|}{|G|}i$.
\end{lem}

The following lemma is a special case of Theorem 1.1 in \cite{glock2020existencedesignsiterativeabsorption}, which ensures the existence of the $F$-decomposition of a typical and divisible $k$-graph.
\begin{lem}[\cite{glock2020existencedesignsiterativeabsorption}, Existence Lemma]\label{Lem:glock_existence}
    Let $f,k\in \mathbb{N}^+$ be fixed with $1< k< f$. Let $c,p\in\mathbb{R}^+$ be fixed with $0<p\le 1$ and $c\le \frac{0.9\left(\frac{p}{2}\right)^h}{q^k4^q}$, where $q:=2f\cdot f!$ and $h:=2^k{q+k \choose k}$. Let $n\in \mathbb{N}^+$ be sufficiently large. Let $F$ be an $f$-vertex $k$-graph and $G$ be an $n$-vertex $k$-graph which is $(c,h,p)$-typical. Then $G$ has an $F$-decomposition if and only if it is $F$-divisible.
\end{lem}

Now we give the lower bound of Theorem \ref{Thm:count_decom}.
\begin{thm}[Lower bound]\label{Thm:lower_decom}
    Let $\zeta,p\in \mathbb{R}^+$ be fixed with $0<p\le 1$, and let $h,k,f\in\mathbb{N}^+$ be fixed with $1< k< f$ and $ h\ge 2^k{2f\cdot f!+k\choose k}.$ Then for every fixed $\beta\in\mathbb{R}^+$ with \(0<\beta<1\), the following holds for all sufficiently large $n\in\mathbb{N}^+$. Let $F$ be an $f$-vertex $k$-graph, and $G$ be an $n$-vertex $F$-divisible $k$-graph that is $\left(n^{-\zeta},h,p\right)$-typical. Then $G$ has at least $n^{\frac{f-k-\beta}{|F|}|G|}$ distinct $F$-decompositions.
\end{thm}
\begin{proof}
    Fix \(\beta\in\mathbb R^+\) with \(0<\beta<1\). Let $\zeta,p\in \mathbb{R}^+$ be fixed with $0<p\le 1$, and let $h,k,f\in\mathbb{N}^+$ be fixed with $1< k< f$ and $ h\ge 2^k{2f\cdot f!+k\choose k}.$ Choose constants $\alpha,\varepsilon,\zeta'\in\mathbb{R}^+$ with $0<\zeta',\varepsilon\ll \zeta$ and $0<\alpha<\frac{\beta}{4(f-k)}$. Note that $h\ge 2^k{2f\cdot f!+k\choose k}\ge \Delta_1(F)$. Set $N:=\frac{|G|}{|F|}$, $p(i)=1-\frac{i}{N}$ and \(m:=\left\lfloor N (1 - \alpha)\right\rfloor\). Note that \(N\) is an integer since \(G\) is \(F\)-divisible. Since \(G\) is \(\left(n^{-\zeta},h,p\right)\)-typical, for every \(S\in\binom{V(G)}{k-1}\), we have \(|N_G(S)|=\left(1\pm n^{-\zeta}\right)pn\), and hence $|G|=\frac1k\sum_{S\in\binom{V(G)}{k-1}}|N_G(S)| =\frac1k\left(1\pm n^{-\zeta}\right)pn\binom n{k-1}=\Theta(n^k)$. Thus \(N=\Theta(n^k)\).

    For an ordered \(F\)-packing \(\mathcal P=(F_0,\ldots,F_{m-1})\) of $G$ with length $m$, write \(G_i(\mathcal P):=G\setminus\bigcup_{j=0}^{i-1}F_j\) for \(i\in [m]\) and $G_0(\mathcal P):=G$. We say that \(\mathcal P\) is \emph{good} if \(G_i(\mathcal P)\) is \(\left(n^{-\zeta'},h,pp(i)\right)\)-typical for every \(i\in\{0,\dots,m\}\). Let \(\mathcal Q\) be the set of all good ordered \(F\)-packings of $G$ with length \(m\). If \(\mathcal P\in\mathcal Q\), then by Proposition~\ref{prop:typical-count}, for every \(i\in\{0,\dots,m-1\}\),
    \[|\mathcal F(G_i(\mathcal P))| = \frac{1}{\operatorname{aut}_F}\left(1\pm 2fn^{-\zeta'}\right) (pp(i))^{|F|}n^f,\]
    which is $\Omega\left(n^f\right)$, since $p(i)\ge p(m)\ge \alpha$.

    Now consider the random removal process of \(G\) with respect to $\mathcal{F}(G)$. Since $m = \lfloor N(1 - \alpha) \rfloor \leq \lfloor N(1 - n^{-\varepsilon}) \rfloor,$ by Lemma~\ref{lem:Random_removal2}, we have that with probability at least $1 - e^{-n^{\Omega(1)}},$ the process will not terminate before step $m$, and for all $i\in\{0,\dots, m\}$, $G_i$ is $\left(n^{-\zeta'},h,pp(i)\right)$-typical. Let $\mathcal{R}$ denote the random ordered \(F\)-packing of $G$ generated by the first \(m\) steps of the random removal process. It follows that $\mathbb{P}[\mathcal{R}\in \mathcal{Q}]\ge 1 - e^{-n^{\Omega(1)}}$. Hence,
    \[1 - e^{-n^{\Omega(1)}}\le \mathbb{P}[\mathcal{R} \in \mathcal{Q}]=\sum_{\mathcal{P}\in \mathcal{Q}}\mathbb{P}[\mathcal{R}=\mathcal{P}]= \sum_{\mathcal{P}\in \mathcal{Q}}\frac{1}{\prod_{i=0}^{m-1} \left|\mathcal{F}(G_i(\mathcal P))\right|}\le |\mathcal{Q}|\left(C_1n^{f}\right)^{-m},\]
    for some constant $C_1\in\mathbb{R}^+$. And thus $|\mathcal{Q}|\ge\left(1 - e^{-n^{\Omega(1)}}\right)\left(C_1n^f\right)^m$. Let \(\mathcal U\) be the set of unordered \(F\)-packings of $G$ with length $m$ whose members admit at least one good ordering. Since every unordered \(F\)-packing of size \(m\) has at most \(m!\) orderings, we have
    \[|\mathcal U|\ge \frac{|\mathcal Q|}{m!} \ge \frac{\left(1 - e^{-n^{\Omega(1)}}\right)\left(C_1n^f\right)^m}{m!}.\]
    As \(m\le N=O(n^k)\), we have \(m!\le m^m\le C_2^m n^{km}\) for some constant \(C_1\le C_2\in\mathbb R^+\). Therefore
    \[|\mathcal U| \ge \left(1 - e^{-n^{\Omega(1)}}\right) \left(\frac{C_1}{C_2}n^{f-k}\right)^m = n^{\left(f-k-O\left(\frac{1}{\log n}\right)\right)m} = n^{\left((1-\alpha)(f-k)-O\left(\frac{1}{\log n}\right)\right)N}.\]

    We now show that every member of \(\mathcal U\) can be extended to an \(F\)-decomposition of \(G\). Fix \(\mathcal P\in\mathcal U\) and let \(H:=G\setminus\bigcup_{F\in\mathcal P}F\). By the definition of \(\mathcal U\), the $k$-graph \(H\) is \(\left(n^{-\zeta'},h,pp(m)\right)\)-typical. Let $h_0:=2^k{2f\cdot f!+k\choose k}$ and choose a constant $c\in \mathbb{R}^+$ such that $c\le \frac{0.9\left({p\alpha}/{2}\right)^{h_0}} {(2f\cdot f!)^k4^{2f\cdot f!}}.$ Since $h\ge h_0$, $pp(m)=p\alpha+o(1)$ and $n^{-\zeta'}=o(1)$, the $k$-graph $H$ is $(c,h_0,p\alpha)$-typical. Also, since \(G\) is \(F\)-divisible and \(\mathcal P\) is an \(F\)-packing, the $k$-graph \(H\) is \(F\)-divisible. Thus \(H\) satisfies the hypothesis of Lemma~\ref{Lem:glock_existence}, and so \(H\) has an \(F\)-decomposition. Hence every \(\mathcal P\in\mathcal U\) can be extended to an \(F\)-decomposition of \(G\).

    It remains only to account for possible over-counting. Every \(F\)-decomposition of \(G\) contains exactly \(N\) copies of \(F\), and hence contains at most \(\binom Nm\) members of \(\mathcal U\) as sub-packings. Therefore, the number of distinct $F$-decompositions of $G$ is at least
    \[\frac{|\mathcal U|}{\binom Nm}\ge n^{\left((1-\alpha)(f-k)-O\left(\frac{1}{\log n}\right)\right)N}\ge n^{\frac{f-k-\beta}{|F|}|G|},\]
    where the first inequality holds since \(\binom Nm\le 2^N=n^{O\left(\frac{N}{\log n}\right)}\), and the second inequality holds since $\alpha<\frac{\beta}{4(f-k)}$. This completes the proof.
\end{proof}

The upper bound is stated as follows.
\begin{thm}[Upper bound]\label{Thm:upper_decom}
    Let $k,f\in\mathbb{N}^+$ be fixed with $1< k< f$ and let $n\in \mathbb{N}^+$ be sufficiently large. Suppose $F$ is an $f$-vertex $k$-graph and $G$ is an $n$-vertex $F$-divisible $k$-graph. Then the number of distinct $F$-decompositions of $G$ is at most $n^{\frac{f-k+o_n(1)}{|F|}|G|}$.
\end{thm}
\begin{proof}
    Set $N:=\frac{|G|}{|F|}$. For every edge $e\in E(G)$, recall that $\mathcal{F}_e(G)$ is the collection of copies of $F$ containing $e$ in $G$. Then there exists a constant $C>0$ such that $|\mathcal{F}_e(G)|\le Cn^{f-k}$ for all $e\in E(G)$. We claim the number of distinct $F$-decompositions of $G$ is at most $\left(Cn^{f-k}\right)^N$. Indeed, order the edges of $G$ as $e_1,\dots, e_{|G|}$ and consider the following greedy process according to this order. At the first step, we take $e_{i_1}:=e_1$, select a copy of $F$ from $\mathcal{F}_{e_{i_1}}(G)$, and denote it by $F_{e_{i_1}}$. For $j\ge 2$, let $e_{i_j}$ be the edge in $G\setminus\bigcup_{\ell=1}^{j-1} F_{e_{i_\ell}}$ with the smallest index. Choose a copy of $F$ in $\mathcal{F}_{e_{i_j}}\left(G\right)$ that is edge-disjoint from $\bigcup_{\ell=1}^{j-1} F_{e_{i_\ell}}$, and denote it by $F_{e_{i_j}}$. Since each selected copy of $F$ is edge-disjoint, this process will continue for at most $N$ steps. At each step, the number of available choices for $F$ is at most $Cn^{f-k}$. Hence, the total number of possible outcomes $\left(F_{e_{i_1}},\ldots,F_{e_{i_N}}\right)$ of this process is at most $\left(Cn^{f-k}\right)^N=n^{\frac{f-k+o_n(1)}{|F|}|G|}.$

    Finally, we assert that every $F$-decomposition can be generated by the above greedy process, so that $n^{\frac{f-k+o_n(1)}{|F|}|G|}$ is an upper bound for the number of distinct $F$-decompositions of $G$. Let $\mathcal{F}^*=\{F_1^*,\ldots,F_N^*\}$ be an arbitrary $F$-decomposition of $G$. We show inductively that it can be generated by the greedy process. At step $j\in [N]$, let $e_{i_j}^*$ be the edge not covered by the previously selected copies with the smallest index. Since $\mathcal{F}^*$ is an $F$-decomposition, there is a unique member of $\mathcal{F}^*$ containing $e_{i_j}^*$. Choose this copy as $F_{e_{i_j}^*}$. Since the members of $\mathcal{F}^*$ are pairwise edge-disjoint, this choice is edge-disjoint from all previously selected copies and hence is valid.
\end{proof}

Then Theorem \ref{Thm:count_decom} follows directly from the lower bound and the upper bound.

\subsection{Proof of Lemma \ref{lem:Random_removal2}}
The key idea is to use Freedman's inequality to show that, throughout the random removal process, the sizes of the common neighbourhoods behave as in a random $k$-graph with the corresponding edge density. To this end, we establish the boundedness and trend of the relevant random variables and estimate their total conditional variances. We also introduce a stopping time to keep the error terms under control, so that the stopped processes are supermartingales to which Freedman's inequality can be applied.
\begin{proof}[Proof of Lemma \ref{lem:Random_removal2}]
    Let $ k, f,h \in \mathbb{N}^+$ be fixed with $1<k< f$ and $h\ge \Delta_1(F)$. Let $\varepsilon,p,\zeta,\zeta'\in\mathbb{R}^+$ be fixed with $0<\zeta',\varepsilon\ll \zeta$ and $0<p\le1$. Choose a constant $\eta\in \mathbb R^+$ such that $\zeta',\varepsilon\ll\eta\ll \zeta$. Since an $\left(n^{-\zeta},h,p\right)$-typical $k$-graph is also $\left(n^{-\zeta_0},h,p\right)$-typical for every $0<\zeta_0\le\zeta$, by decreasing $\zeta$ if necessary, we may assume that $0<\zeta<1$. Set $m: = \left\lfloor\frac{|G|}{|F|} (1 - n^{-\varepsilon})\right\rfloor$ and $p(i)=1 - \frac{|F|}{|G|}i$. Since \(G\) is \(\left(n^{-\zeta},h,p\right)\)-typical, $|G|=\frac{1}{k} \sum_{S\in\binom{V\left(G\right)}{k-1}}d_G\left(S\right)=\frac{1}{k}\binom n{k-1} \left(1\pm n^{-\zeta}\right)pn.$ In particular, $|G|=\Theta\left(n^k\right).$

    \vspace{0.8 em}
    \noindent\textbf{Setup of the tracking random variables.}
    We use the same notation for the random removal process of $G$ with respect to $\mathcal{F}(G)$ as in Definition \ref{random process}. By convention, we take $\bigcap_{S\in\emptyset}N_G(S)=V(G)$. For every \(A\subseteq \binom{V\left(G\right)}{k-1}\) with \(|A|\le h\) and $i\in\{0,\dots,m\}$, define $X_A\left(i\right):= \left| \bigcap_{S\in A}N_{G_i}(S) \right|$, $M_A(i):= \left| \bigcap_{S\in A}N_{G}\left(S\right) \right|p(i)^{|A|},$ and $e_A(i):=n^{-\eta}\left| \bigcap_{S\in A}N_{G}\left(S\right) \right|p(i)^{|A|-(8f+1)h}$. Since $G$ is $\left(n^{-\zeta},h,p\right)$-typical, we have $\left| \bigcap_{S\in A}N_{G}\left(S\right) \right|=\big(1\pm n^{-\zeta}\big)p^{|A|}n.$ Since \(p(i)\ge p(m)=(1+o(1))n^{-\varepsilon}\) and $\eta\gg \varepsilon$, we have $e_A(i)=o(M_A(i)).$

    We will show that with probability at least \(1 - e^{-n^{\Omega(1)}}\), for every \(A\subseteq \binom{V\left(G\right)}{k-1}\) with \(|A|\le h\) and $i\in\{0,\dots,m\}$, $X_A\left(i\right)=M_A(i)\pm e_A(i),$ which implies $X_A\left(i\right)=\left(1\pm n^{-\zeta'}\right)\left(pp\left(i\right)\right)^{|A|}n $ by $\zeta'\ll\eta\ll\zeta$. This is precisely the assertion that \(G_i\) is \(\left(n^{-\zeta'},h,pp\left(i\right)\right)\)-typical for every \(i\in\{0,\dots,m\}\). Note that the case $A=\emptyset$ is trivial, since $X_\emptyset(i)=M_\emptyset(i)=n$ for every $i\in\{0,\dots,m\}$. Hence, in what follows, we may assume $|A|\ge 1$.

    For every \(A\subseteq \binom{V\left(G\right)}{k-1}\) with \(1\le|A|\le h\) and $i\in\{0,\dots,m\}$, define $X_A^+\left(i\right) :=X_A\left(i\right)-\left(M_A(i)+e_A(i)\right),$ and $X_A^-(i):=M_A(i)-(e_A(i)+X_A(i)).$ Since \(G\) is \(\left(n^{-\zeta},h,p\right)\)-typical, we have $X_A^+\left(0\right)=X_A^-(0)=-e_A(0)=-n^{-\eta}\left| \bigcap_{S\in A}N_G(S)\right|\le 0.$ Then it suffices to show with probability at least \(1 - e^{-n^{\Omega(1)}}\), for every \(A\subseteq \binom{V\left(G\right)}{k-1}\) with \(1\le|A|\le h\) and every \(i\in[m]\), $X_A^+(i)\le 0$ and $X_A^-(i)\le 0.$
    
    \vspace{0.8 em}
    \noindent\textbf{Checking the boundedness of random variables.} We have the following boundedness claim.
    \begin{claim}[Boundedness]\label{claim: boundedness}
        For every \(A\subseteq \binom{V\left(G\right)}{k-1}\) with \(1\le|A|\le h\) and $i\in\{0,\dots,m-1\}$, we have $ \left|X_A^+\left(i+1\right)-X_A^+\left(i\right)\right|=O\left(1\right),$ and $\left|X_A^-\left(i+1\right)-X_A^-\left(i\right)\right|=O\left(1\right). $
    \end{claim}
    \begin{proof}

        Fix \(A\subseteq \binom{V\left(G\right)}{k-1}\) with \(1\le|A|\le h\). For each $i\in\{0,\ldots,m-1\}$, in the process from $G_i$ to $G_{i+1}$, we delete one copy of $F$, and hence exactly $|F|$ edges of $G_i$. Each deleted edge can destroy at most \(k\) pairs $(S,v)$ with \(S\in A\) and $v\in V(G)$. Thus one step can change \(X_A\left(i\right)\) by at most \(k|F|\). Consequently, $|X_A(i+1)-X_A(i)|=O(1)$ for every \(A\subseteq \binom{V\left(G\right)}{k-1}\) with \(1\le|A|\le h\) and $i\in\{0,\dots,m-1\}$. Moreover, for every \(A\subseteq \binom{V\left(G\right)}{k-1}\) with \(1\le|A|\le h\) and $i\in\{0,\dots,m-1\}$, since \(p(i+1)=p(i)\left(1-\frac{|F|}{p(i)|G|}\right)\), we have
        \begin{align}\label{inq:M-diff-lower}
            M_A(i)-M_A(i+1) &=\left|\bigcap_{S\in A}N_G(S)\right|p(i)^{|A|} \left(1-\left(1-\frac{|F|}{p(i)|G|}\right)^{|A|}\right),
        \end{align}
        which is at most $M_A(i)|A|\frac{|F|}{p(i)|G|}$, and at least $ M_A(i)\left(|A|\frac{|F|}{p(i)|G|}-O\left(\frac{1}{p(i)^2|G|^2}\right)\right), $ using the binomial expansion. Similarly, for every \(A\subseteq \binom{V\left(G\right)}{k-1}\) with \(1\le|A|\le h\) and $i\in\{0,\dots,m-1\}$, we have
        \begin{align}\label{inq:E-diff-lower}
            e_A(i+1)-e_A(i)=n^{-\eta}\left|\bigcap_{S\in A}N_G(S)\right|p(i)^{|A|-(8f+1)h}\left(\left(1-\frac{|F|}{p(i)|G|}\right)^{|A|-(8f+1)h} -1\right),
        \end{align}
        which is at least $ e_A(i) \frac{8f|A||F|}{p(i)|G|}, $ and at most $e_A(i)\Big( \frac{((8f+1)h-|A|)|F|}{p(i)|G|}+O\left(\frac{1}{p(i)^2|G|^2}\right)\Big)=O\left(\frac{e_A(i)}{p(i)|G|}\right)$, using the binomial expansion. Hence, for $*\in \{+,-\}$, using $e_A(i)\le M_A(i)\le n$, \(p(i)\ge n^{-\varepsilon}\) and \(|G|=\Theta(n^k)\), for every \(i\in\{0,\dots,m-1\}\) and every \(A\subseteq \binom{V\left(G\right)}{k-1}\) with \(1\le|A|\le h\), we have
        \begin{align}\label{inq:X*-diff}
            \left|X_A^*\left(i+1\right)-X_A^*\left(i\right)\right|&\le \left|X_A\left(i+1\right)-X_A\left(i\right) \right|+|M_A(i+1)-M_A(i)|+|e_A(i+1)-e_A(i)|\notag\\
            &=O\left(1\right)+O\left(M_A(i)\frac{1}{p(i)|G|}+e_A(i)\frac{1}{p(i)|G|}\right)\\
            &=O\left(1\right).\notag\qedhere
        \end{align}
    \end{proof}
    
    \vspace{0.8 em}
    \noindent\textbf{Setup of the stopping time and filtration.}
    For each outcome of the random process, define a \emph{stopping time} $\tau$ to be the first step $i\in\{0,\ldots,m\}$ at which there exists $A\subseteq\binom{V(G)}{k-1}$ with $1\le |A|\le h$ such that $|X_A(i)-M_A(i)|>e_A(i)$. If no such step exists, then set $\tau=m+1$. We freeze the process after time $\tau$. More precisely, for every $i\in\{0,\ldots,m\}$ and every \(A\subseteq \binom{V\left(G\right)}{k-1}\) with \( 1\le |A|\le h\), define $\widetilde{G}_i:=G_{i\wedge\tau}$, $\widetilde X^+_{A}(i):=X^+_{A}(i\wedge\tau)$ and $\widetilde X^-_{A}(i):=X^-_{A}(i\wedge\tau)$. Thus, after the stopping time, both the state of the random removal process and all the corresponding random variables remain unchanged. For simplicity, we omit the tildes and continue to use the original notation for the stopped process and the stopped random variables.

    Let \(\left(\Sigma_i\right)_{i=0}^m\) be the natural filtration of the random removal process. More precisely, $ \Sigma_0:=\{\emptyset,\Omega\}$ and $\Sigma_i:=\sigma(F_0,\ldots,F_{i-1})$ for every $i\in[m]$, where $\Omega$ is the sample space of the stopped random removal process, and $\sigma(F_0,\ldots,F_{i-1})$ denotes the smallest $\sigma$-algebra on $\Omega$ with respect to which the random variables $F_0,\ldots,F_{i-1}$ are all measurable. Recall that $F_{i-1}$ denotes the copy of $F$ selected at the $i$-th step of the random removal process. Thus, $\Sigma_i$ records precisely the information revealed during the first $i$ steps of the process. For every $i\in\{0,\ldots,m\}$, conditioned on $\Sigma_i$, the graph $G_i$ is fixed. In particular, $X^+_A(i)$ and $X^-_A(i)$ are $\Sigma_i$-measurable for every $A\subseteq {V(G) \choose k-1}$ with $1\le|A|\le h$.
    
    \vspace{0.8 em}
    \noindent\textbf{Checking the trend of random variables.} We have the following trend claim.
    \begin{claim}[Trend]\label{claim:trend}
        For every $A\subseteq {V(G) \choose k-1}$ with $1\le|A|\le h$ and every $i\in\{0,\ldots,m-1\}$, we have $\mathbb{E}[X^+_A(i+1)\mid \Sigma_i] \le X^+_A(i),$ and $\mathbb{E}[X^-_A(i+1)\mid \Sigma_i] \le X^-_A(i).$
    \end{claim}
    \begin{proof}
        Since we consider the conditional expectation, we first fix a realization of $\Sigma_i$. If $i\ge \tau$ in this realization, then the sequence is frozen, so we have $\mathbb{E}[ X^+_A(i+1)\mid \Sigma_i ] = X^+_A(i)$ and $\mathbb{E}[ X^-_A(i+1)\mid \Sigma_i ] = X^-_A(i)$. Hence the conclusion follows. So it is sufficient to consider when $0 \le i\le \tau-1$ in this realization. Then for every \(A\subseteq\binom{V(G)}{k-1}\) with \(1\le |A|\le h\), we have
        \begin{align}\label{inq:X_A}
            X_{A}(i)= M_{A}(i)\pm e_{A}(i)= \big(1\pm n^{-\zeta}\big)\left(1\pm n^{-\eta}p(i)^{-(8f+1)h}\right)(pp(i))^{|A|}n.
        \end{align}
        Consequently, $G_i$ is $(c_i(n),h,pp(i))$-typical, where $c_i(n):=(1+o(1))n^{-\eta}p(i)^{-(8f+1)h}$, using $\eta\ll \zeta$. Moreover, by $\varepsilon\ll \eta$, we have $\frac{f-1}{(pp(i))^hn} =o\left(c_i(n)\right)$. Then Proposition~\ref{prop:typical-count} gives
        \begin{align}\label{inq:F}
            |\mathcal F(G_i)|= \frac{1}{\operatorname{aut}_F} \left(1\pm 2c_i(n)\right)^f(pp(i))^{|F|}n^f,
        \end{align}
        and for every edge $e\in E(G_i)$, we have
        \begin{align}\label{inq:F_e}
            |\mathcal F_e(G_i)| = \frac{|F|\cdot k!}{\operatorname{aut}_F} \left(1\pm 2c_i(n)\right)^{f-k}(pp(i))^{|F|-1}n^{f-k},
        \end{align}
        where $\operatorname{aut}_{F}$ is the number of automorphisms of $F$. Inequality (\ref{inq:F}) implies that $|\mathcal F(G_i)|>0$ and thus the random process will not terminate before step $\tau$.

        For each \(v\in\bigcap_{S\in A}N_{G_i}\left(S\right)\), set $E_v:=\left\{S\cup\{v\}:S\in A\right\}.$ Then \(E_v\subseteq E\left(G_i\right)\) and \(\left|E_v\right|=|A|\). The vertex \(v\) does not belong to \(\bigcap_{S\in A}N_{G_{i+1}}\left(S\right)\) precisely when the copy of \(F\) chosen at step \(i+1\) contains at least one edge from \(E_v\). Therefore,
        \begin{align}\label{inq:pr_v}
            \mathbb{P}\left[ v\notin\bigcap_{S\in A}N_{G_{i+1}}\left(S\right) \mid\Sigma_i\right]=&\frac{\left|\bigcup_{e\in E_v}\mathcal F_e\left(G_i\right)\right|}{\left|\mathcal F\left(G_i\right)\right|} =\sum_{e\in E_v}\frac{\left|\mathcal F_e\left(G_i\right)\right|}{\left|\mathcal F\left(G_i\right)\right|}\pm\sum_{\substack{e,e'\in E_v\\
            e\ne e'}}\frac{\left|\mathcal F_{e,e'}\left(G_i\right)\right|} {\left|\mathcal F\left(G_i\right)\right|},
        \end{align}
        where \(\mathcal F_{e,e'}\left(G_i\right)\) denotes the family of copies of \(F\) in \(G_i\) containing both \(e\) and \(e'\), and the second equation holds by the inclusion-exclusion principle. By double counting, we have $\sum_{e\in E\left(G_i\right)} \left|\mathcal F_e\left(G_i\right)\right| =|F|\left|\mathcal F\left(G_i\right)\right|.$ By inequality (\ref{inq:F_e}), $\left|\mathcal F_e\left(G_i\right)\right| =\left(1\pm 4c_i(n)\right) \frac{|F|\left|\mathcal F\left(G_i\right)\right|} {|G_i|}$ for every \(e\in E\left(G_i\right)\). Hence,
        \begin{align}\label{inq:pr_vt1}
            \frac{\left|\mathcal F_e\left(G_i\right)\right|} {\left|\mathcal F\left(G_i\right)\right|} =\left(1\pm 4c_i(n)\right) \frac{|F|}{|G_i|}.
        \end{align}
        Moreover, for distinct \(e,e'\in E_v\), \(\left|e\cup e'\right|\ge k+1\). After fixing \(e\) and \(e'\), there are at most \(O\left(n^{f-k-1}\right)\) possible choices for a copy of \(F\) containing both of them, hence $\left|\mathcal F_{e,e'}\left(G_i\right)\right| =O\left(n^{f-k-1}\right).$ By inequality (\ref{inq:F}), $\left|\mathcal F\left(G_i\right)\right| =\Omega\left( n^{f-|F|\varepsilon}\right).$ Since \(|E_v|=|A|\le h\), it follows that
        \begin{align}\label{inq:pr_vt2}
            \sum_{\substack{e,e'\in E_v\\
            e\ne e'}} \frac{\left|\mathcal F_{e,e'}\left(G_i\right)\right|} {\left|\mathcal F\left(G_i\right)\right|} =O\left(n^{-k-1+|F|\varepsilon}\right).
        \end{align}
        Then by inequalities (\ref{inq:pr_v}), (\ref{inq:pr_vt1}), (\ref{inq:pr_vt2}), we have
        \begin{align}\label{inq:v-pr}
            \mathbb{P}\left[ v\notin\bigcap_{S\in A}N_{G_{i+1}}\left(S\right) \mid\Sigma_i \right]=\left(1\pm 5c_i(n)\right) \frac{|A||F|}{|G_i|}.
        \end{align}
        Summing over vertices $v$ in \(\bigcap_{S\in A}N_{G_i}\left(S\right)\), using \(|G_i|=p\left(i\right)|G|\) and inequality (\ref{inq:X_A}), we have
        \begin{align}\label{Eq:EX_Adiff}
            \mathbb{E}\left[X_A(i)-X_A(i+1) \mid\Sigma_i\right]&= \sum_{v\in\bigcap_{S\in A}N_{G_i}(S)}\mathbb{P}\left[ v\notin\bigcap_{S\in A}N_{G_{i+1}}\left(S\right) \mid\Sigma_i \right]\notag\\
            &= X_A(i) \left(1\pm 5c_i(n)\right) \frac{|A||F|}{|G_i|}\\
            &= \left(1\pm 6c_i(n)\right) \frac{|A||F|}{p\left(i\right)|G|} M_A\left(i\right).\notag
        \end{align}
        Hence, by inequalities \eqref{inq:M-diff-lower} and \eqref{inq:E-diff-lower}, we have
        \[\begin{aligned} \mathbb E[X_A^+(i+1)-X_A^+(i)\mid\Sigma_i] &=-\mathbb E[X_A(i)-X_A(i+1)\mid\Sigma_i] +M_A(i)-M_A(i+1) +e_A(i)-e_A(i+1) \\
            &\le\left( -1+6fc_i(n)+1 -8fc_i(n)\right)\frac{|A||F|}{p(i)|G|}M_A(i) \\
            &\le 0, \end{aligned}\]
        and using $c_i(n)\ge n^{-1}$, we have
        \[\begin{aligned} \mathbb E[X_A^-(i+1)-X_A^-(i)\mid\Sigma_i] =& \mathbb E[X_A(i)-X_A(i+1)\mid\Sigma_i] -\left(M_A(i)-M_A(i+1)\right) -\left(e_A(i+1)-e_A(i)\right)\\
            \le& \left(1+6fc_i(n) -1+O(n^{-k+\varepsilon}) -8fc_i(n)\right) \frac{|A||F|}{p(i)|G|}M_A(i)\\
            \le&0. \end{aligned}\]
        Thus, for each $0\le i\le\tau-1$, we have $\mathbb{E}\left[ X_A^+\left(i+1\right)\mid\Sigma_i \right] \le X_A^+\left(i\right)$ and $ \mathbb{E}\left[ X_A^-\left(i+1\right)\mid\Sigma_i \right] \le X_A^-\left(i\right)$.
    \end{proof}
    Therefore, after freezing, \(\left(X_A^+(i)\right)_{i=0}^{m}\) and \(\left(X_A^-(i)\right)_{i=0}^{m}\) are supermartingales for every $A\subseteq {V(G) \choose k-1}$ with $1\le|A|\le h$.
    
    \vspace{0.8 em}
    \noindent\textbf{Estimating the total conditional variances.}
    Note that for $*\in \{+,-\}$, for every $A\subseteq {V(G) \choose k-1}$ with $1\le|A|\le h$ and every $i\in\{0,\dots,m-1\}$,
    \begin{align*}
        \operatorname{Var}\left[X_A^*(i+1)\mid \Sigma_i \right]&=\mathbb{E}\left[\left(X_A^*(i+1)\right)^2\mid \Sigma_i \right]-\left(\mathbb{E}[X_A^*(i+1)\mid \Sigma_i]\right)^2 \\
        &\le \mathbb{E}\left[\left(X_A^*(i+1)\right)^2\mid \Sigma_i \right]- 2\mathbb{E}\left[X_A^*(i+1)\mid \Sigma_i \right]X_A^*(i)+\left(X_A^*(i)\right)^2 \\
        &= \mathbb{E}\left[\left| X_A^*(i+1)-X_A^*(i)\right| ^2\mid \Sigma_i \right].
    \end{align*}
    By equality (\ref{Eq:EX_Adiff}), for every $A\subseteq {V(G) \choose k-1}$ with $1\le|A|\le h$ and for every $i\in\{0,\dots,m-1\}$ and $i\le \tau-1$,
    \[\begin{aligned} \mathbb{E}\left[X_A(i)-X_A(i+1) \mid\Sigma_i\right] = X_A(i) \left(1\pm 5fc_i(n)\right) \frac{|A||F|}{|G_i|} =O\left(\frac{X_A(i)}{p(i)|G|} \right). \end{aligned}\]
    By Claim \ref{claim: boundedness}, for every $A\subseteq {V(G) \choose k-1}$ with $1\le|A|\le h$ and every $i\in\{0,\dots,m-1\}$, $\left|X_A^*\left(i+1\right)-X_A^*\left(i\right) \right|=O\left(1\right).$ Then by inequality \eqref{inq:X*-diff}, for every $A\subseteq {V(G) \choose k-1}$ with $1\le|A|\le h$ and for every $i\in\{0,\dots,m-1\}$ and $i\le \tau-1$, we have
    \begin{align*}
        \operatorname{Var}\left[ X_A^*\left(i+1\right)\mid\Sigma_i \right] &\le \mathbb{E}\left[\left| X^*_A(i+1)-X^*_A(i)\right| ^2\mid\Sigma_i\right] \\
        &\le \max_i\left|X^*_A(i+1)-X^*_A(i)\right|\mathbb{E}\left[\left| X_A^*\left(i+1\right)-X_A^*\left(i\right) \right|\mid \Sigma_i\right] \\
        &\le O\left(1\right)\left( \mathbb{E}\left[ \left|X_A\left(i+1\right)-X_A\left(i\right) \right|\mid\Sigma_i\right] +|M_A(i+1)-M_A(i)|+|e_A(i+1)-e_A(i)|\right) \\
        &=O\left( \frac{X_A(i)+M_A(i)+e_A(i)}{p(i)|G|}\right)\\
        &=O\left(\frac{M_A(i)+e_A(i)}{p(i)|G|}\right).
    \end{align*}
    Consequently, by $e_A(i)\le M_A(i)\le n$, \(p(i)\le 1\) and \(m=O(|G|)\), we have
    \[\begin{aligned} \sum_{i=0}^{m-1} \operatorname{Var}\left[ X_A^*\left(i+1\right)\mid\Sigma_i \right] =O\left( \sum_{i=0}^{m-1} \frac{M_A(i)+e_A(i)}{p(i)|G|} \right)=O\left( n\sum_{i=0}^{m-1}\frac{p(i)^{|A|-1}}{|G|} \right)=O(n). \end{aligned}\]
    
    \vspace{0.8 em}
    \noindent\textbf{Applying the supermartingale concentration.} For $*\in \{+,-\}$,
    apply Lemma \ref{Lem:Freedman} (Freedman's inequality) with $a=-X^*_A(0)=e_A(0)=\Theta\left(n^{1-\eta}\right),$ $b=O\left(1\right),$ $v=O\left( n\right)$. Then we have for fixed \(A\subseteq \binom{V\left(G\right)}{k-1}\) with \(1\le|A|\le h\),
    \[\mathbb{P}\left[\exists\ i \in [m] \text{ s.t. } X^*_A(i)\ge 0\right] \le e^{-\frac{a^2}{2(v+ba)}}=e^{-\Omega\left(n^{1-2\eta} \right) }.\]
    Taking a union bound over $X_A^*(i)$ and over at most $\sum_{j=1}^{h}\binom{\binom{n}{k-1}}{j}=n^{O(1)}$ choices of $A$, we have with probability at least \(1 - e^{-n^{\Omega(1)}}\), for every \(A\subseteq \binom{V\left(G\right)}{k-1}\) with \(1\le|A|\le h\) and every \(i\in[m]\), $X_A^*(i)\le 0$. If $\tau\in [m]$, then by the definition of $\tau$, there exists $A\subseteq\binom{V(G)}{k-1}$ with $1\le |A|\le h$ such that either $X_A^+(\tau)>0$ or $X_A^-(\tau)>0$. Hence, with probability at least $1-e^{-n^{\Omega(1)}}$, we have $\tau=m+1$. Consequently, with probability at least \(1 - e^{-n^{\Omega(1)}}\), the random removal process does not terminate before step $m$, and for every \(i\in\{0,\dots,m\}\), \(G_i\) is \(\left(n^{-\zeta'},h,pp(i)\right)\)-typical. This completes the proof.
\end{proof}

\section{Further Applications}
\label{Sec:UC}

In this section, we apply Theorem \ref{Thm:count_decom} to asymptotically enumerate the number of universal cycles for $k$-subsets of $[n]$. As discussed in the introduction, a universal cycle for $k$-subsets of $[n]$ is precisely a tight Euler tour of the complete $k$-graph $K_{n}^k$. We then extend the result to Euler $\ell$-tours in typical $k$-graphs. We first formally introduce definitions of universal cycles, $\ell$-cycles and Euler $\ell$-tours.

\begin{dfn}[Universal cycle]
    Let $k,n\in \mathbb{N}^+$ with $n>k>1$. A \textbf{universal cycle} for $k$-subsets of $[n]$, denoted by $\mathcal{C}_{n,k}=\left(a_1,a_2,\ldots,a_{n \choose k}\right)$, is a cyclic sequence of $[n]$ such that $\left\{ a_i,a_{i+1},\ldots,a_{i+k-1}\right\} \in {[n] \choose k}$ and
    \[\left\{a_i,a_{i+1},\ldots,a_{i+k-1}\right\}\ne \left\{ a_j,a_{j+1},\ldots,a_{j+k-1}\right\}\]
    for every $i,j\in\left[{n \choose k}\right],i\ne j$, with indices modulo ${n \choose k}$. Two universal cycles are regarded as the same if one can be obtained from the other by a cyclic shift or by reversing the cyclic order.
\end{dfn}

\begin{dfn}[$\ell$-cycles]
    Suppose $k,t,\ell\in \mathbb{N}^+$ satisfy $0<\ell<k< t$ and $(k-\ell) \mid t$. An \textbf{$\ell$-cycle } $C_t^{k,\ell}$ with length $\frac{t}{k-\ell}$ is a $k$-graph on $t$ vertices admitting a cyclic ordering $v_1,\dots,v_t$ such that its edges are precisely the $k$-sets $\{v_{(i-1)(k-\ell)+1},\dots,v_{(i-1)(k-\ell)+k}\}$ for $i\in\left[\frac{t}{k-\ell}\right]$, where indices are taken modulo $t$. When $\ell=k-1$, we simply write $C_t^{k,k-1}$ as $C_t^k$ and call it a \textbf{tight cycle}.
\end{dfn}

\begin{dfn}[$\ell$-walks]
    Suppose $t,k,\ell\in \mathbb{N}^+$ satisfy $0<\ell<k$. Let $G$ be a $k$-graph. An \textbf{$\ell$-walk} in $G$ with length $t$ is a pair $W=\big(\left(v_1,\dots, v_{k+(t-1)(k-\ell)}\right),\left(e_1,\dots, e_t\right)\big)$, where $v_1,\dots, v_{k+(t-1)(k-\ell)}\in V(G)$ may be repeated, while $e_1,\ldots,e_t\in E(G)$ are pairwise distinct and satisfy
    \[e_i=\Big\{v_{(i-1)(k-\ell)+1},\dots, v_{(i-1)(k-\ell)+k}\Big\}\]
    for every $i\in[t]$. When $\ell=k-1$, we refer to an $\ell$-walk as a \textbf{tight walk}.
\end{dfn}
\begin{dfn}[closed $\ell$-walks and Euler $\ell$-tours]
    Suppose $t,k,\ell\in \mathbb{N}^+$ satisfy $0<\ell<k$ and $t(k-\ell)\ge k$. We define a \textbf{closed $\ell$-walk} with length $t$ in $G$ as a pair $W=\big((v_1,\dots,v_{t(k-\ell)}),(e_1,\dots,e_t)\big)$, where $v_1,\dots, v_{t(k-\ell)}\in V(G)$ may be repeated, while $e_1,\ldots,e_t\in E(G)$ are pairwise distinct and satisfy $e_i=\left\{v_{(i-1)(k-\ell)+1},\ldots, v_{(i-1)(k-\ell)+k}\right\}$ for every $i\in[t]$, with indices of the vertices taken modulo $t(k-\ell)$. Two representations of a closed $\ell$-walk are regarded as the same if one can be obtained from the other by cyclically shifting the edge sequence by $q$ positions and the vertex sequence by $q(k-\ell)$ positions for some integer $q$, or by reversing both cyclic orders.

    Write $E(W)=\{e_1, \ldots, e_t\}.$ An \textbf{Euler $\ell$-tour} of $G$ is a closed $\ell$-walk $W$ of $G$ satisfying $E(W)=E(G).$ When $\ell=k-1$, we refer to a closed $\ell$-walk and an Euler $\ell$-tour as a \textbf{tight closed walk}, and a \textbf{tight Euler tour}, respectively.
\end{dfn}

Note that in all cases, $W$ can be uniquely determined by the prescribed step size $k-\ell$ and its vertex sequence with a starting point and an orientation. Thus, for simplicity, we often represent $W$ by a sequence of vertices when $k$ and $\ell$ are fixed.

We now establish a general upper bound on the number of Euler $\ell$-tours, which will also yield the upper bound for universal cycles as a special case.

\begin{thm}[Upper bound]\label{prop:euler_tour_upper_bound}
    Let $k,\ell, n\in \mathbb{N}^+$ be fixed with \(0<\ell<k\). Let \(G\) be a \(k\)-graph on \(n\) vertices. Then the number of distinct Euler \(\ell\)-tours of \(G\) is at most $k!n^{(k-\ell)(|G|-1)}.$
\end{thm}

\begin{proof}
    The case $E(G)=\emptyset$ is trivial, so assume $E(G)\ne \emptyset$. Fix an edge \(e_0\in E(G)\). Since every Euler \(\ell\)-tour contains every edge of \(G\) exactly once, every Euler \(\ell\)-tour contains \(e_0\) in its edge sequence. There are at most \(k!\) possible orders for the \(k\) vertices of \(e_0\). We now choose the remaining entries of the vertex sequence successively in blocks of \(k-\ell\) vertices. Once the first \(i \) edges, where $i\in[|G|-1]$, have been determined, the next edge is obtained by retaining the prescribed ordered overlap of \(\ell\) vertices and choosing the next \(k-\ell\) vertices. Hence there are at most $n^{k-\ell}$ choices at each step. Since there are \(|G|-1\) edges after \(e_0\), the number of distinct Euler \(\ell\)-tours of \(G\) is at most $k!\left(n^{k-\ell}\right)^{|G|-1} = k!n^{(k-\ell)(|G|-1)}.$
\end{proof}

\subsection{Counting universal cycles for $k$-subsets of $[n]$}
In this subsection, we focus on the proof of Theorem \ref{Thm:uc}. The main ingredient is the following lower bound for the number of universal cycles, which follows from the construction of tight Euler tours together with our $F$-decomposition counting result.

\begin{thm}[Lower bound]\label{Thm:uc_lower}
    Let $k \in \mathbb{N}^+$ be fixed with $k>1$. For every fixed $\beta\in\mathbb{R}^+$ with $0<\beta<1$, and all sufficiently large $n\in \mathbb{N}^+ $ satisfying $k\mid {n-1 \choose k-1}$, the number of distinct universal cycles for $k$-subsets of $[n]$ is at least $n^{(1-\beta){n\choose k}}$.
\end{thm}

In \cite{2020Euler}, they give the following initial walk in the Euler tour construction as Lemma \ref{Lemma00}. To adapt our proof, we make it a little stronger as Lemma \ref{Lemma0}. We say a closed $\ell$-walk $W=\big((v_1,\dots,v_{t(k-\ell)}),(e_1,\dots,e_t)\big)$ of $G$ is \emph{spanning} if every ordered $\ell$-tuple $(u_1,\dots,u_{\ell})$ of distinct vertices of $G$ occurs at least once as the overlap of two consecutive edges of $W$ with the order determined by the vertex sequence, i.e. $(u_1,\dots,u_{\ell})=\left(v_{i(k-\ell)+1},\ldots,v_{i(k-\ell)+\ell}\right)$ for some $i\in [t]$, with indices of the vertices modulo $t(k-\ell)$. In the non-closed case, we additionally allow $(u_1, \dots, u_\ell)$ to appear as the first or last $\ell$ vertices of $W$.

\begin{lem}[\cite{2020Euler}]\label{Lemma00}
    Let $h,k\in \mathbb{N}^+$ be fixed with $h\ge k>1$ and $\zeta,p\in \mathbb{R}^+$ be fixed with $0<p\le 1$. Let $n\in \mathbb{N}^+$ be sufficiently large. If $G$ is an $n$-vertex $k$-graph which is $\left(n^{-\zeta},h,p\right)$-typical, then there exists a spanning tight walk ${W}_0$ in $G$ such that $\Delta_{k-1}({W}_0)\le \log^3n$.
\end{lem}

\begin{lem}\label{Lemma0}
    Let $h,k,T\in \mathbb{N}^+$ be fixed with $h\ge k>1$ and $2k\le T$. Let $\zeta,p\in \mathbb{R}^+$ be fixed with $0<p\le 1$. Let $n\in \mathbb{N}^+$ be sufficiently large. Suppose $G$ is an $n$-vertex $k$-graph which is $\left(n^{-\zeta},h,p\right)$-typical. If $k\mid d_G(v)\text{ for every } v\in V(G)$, then there exists a spanning tight closed walk ${W}$ in $G$ such that $\Delta_{k-1}({W})\le 2\log^3n$ and $G\setminus{W}$ is $C_{T}^k$-divisible.
\end{lem}
\begin{proof}
    By Lemma \ref{Lemma00}, there exists a spanning tight walk $W_0=v_1v_2\cdots v_{\ell}$ such that $\Delta_{k-1}(W_0)\le \log^3n$. We extend $W_0$ to a tight closed walk $W$ such that $G\setminus W$ satisfies the divisible condition. Note that $d_{C_T^k}(0)=T$, $d_{C_T^k}(1)=k$, and for every $2\le j\le k-1$, $d_{C_T^k}(j)=1$. Indeed, one may choose a $j$-set containing the two end vertices of an edge of $C_T^k$. Since $T\ge 2k$, this $j$-set is contained in exactly one edge of $C_T^k$. Hence a $k$-graph $G$ is $C_{T}^k$-divisible if $T\mid|G|$ and $k\mid d_G(v)\text{ for every } v\in V(G)$. Note that the graph $G$ satisfies the degree divisibility condition. Moreover, any tight closed walk, considered as a subgraph, also satisfies the degree divisibility condition, since the degree of each vertex within the walk is a multiple of $k$. Thus $G\setminus W$ satisfies the degree divisibility condition. It remains to verify the edge divisibility condition $T \mid |G \setminus W|$.

    Now we choose $k\le \ell'\le k+ T-1$ such that $0\equiv|G|-\ell-\ell^{\prime}\pmod {T} $, and claim that we can find distinct vertices $v_{\ell+1},\ldots,v_{\ell+\ell'}\in V(G)\setminus \{v_1,\ldots,v_{k-1},v_{\ell-k+2},\ldots,v_\ell\}$ one by one such that $\{v_i,v_{i+1},\ldots,v_{i+k-1}\} $ are distinct edges in $E(G\setminus W_0)$ for every $\ell-k+2\le i\le \ell+\ell'$ with indices modulo $\ell+\ell'$. Note that we need $\ell^{\prime}\ge k$ since if $v_1=v_\ell$, we need at least $k$ new vertices to form distinct edges. Actually, in the worst case, the vertex $v_{\ell+\ell'}$ must be a common neighbor of $k$ different $(k-1)$-sets in $G\setminus W_0$. Since $h\ge k$, at each step, there are at least
    \[\left(1-n^{-\zeta}\right)p^kn-k\log^3n-2(k-1)-\ell'=\Omega(n)\]
    vertices that can be chosen. Let $W=(v_1,v_2,\ldots,v_{\ell+\ell'})$. Then $W$ is a spanning tight closed walk and $\Delta_{k-1}(W)\le 2\log^3n$, since only finitely many edges are added:
    \[|W|-|W_0|=\ell+\ell'-(\ell-k+1)=\ell'+k-1\le 2k+T-1.\]
    By the choice of $\ell'$, we have $T \mid (|G|-\ell-\ell')$ and $|G|-\ell-\ell' = |G \setminus W|$. Hence $G\setminus W $ is $C_{T}^k$-divisible.
\end{proof}

Then we give the lower bound for the number of distinct universal cycles for $k$-subsets of $[n]$.

\begin{proof}[Proof of Theorem \ref{Thm:uc_lower}]
    Fix $k\in\mathbb{N}^+$ and $\beta\in\mathbb{R}^+$ with $k>1$ and $0<\beta<1$. Choose a constant $T\in\mathbb{N}^+$ such that $T \ge \max\left\{2k,\frac{k+2}{\beta}\right\}$. Let $n\in \mathbb{N}^+$ be sufficiently large subject to $k\mid {n-1\choose k-1}$. Note that the complete $k$-graph $K_n^k$ is $\left(\frac{h(k-1)}{n},h,1\right)$-typical for arbitrary fixed $h\in \mathbb{N}^+$. Let $W$ be the spanning tight closed walk of $K_n^k$ given by Lemma \ref{Lemma0}, so that $\Delta_{k-1}(W)\le 2\log^3n$ and $K_n^k\setminus W$ is $C_{T}^k$-divisible. Define $G':=K_n^k\setminus W$. Then for any fixed $h\in\mathbb{N}^+$ and for every $A\subseteq {[n]\choose k-1}$ with $ |A|\le h$, we have
    \[\left|\bigcap_{S\in A}N_{G'}(S) \right| \ge n-(k-1)|A|-|A|\cdot 2\log ^3n\ge \left( 1-n^{-\zeta}\right) n\]
    for every fixed $\zeta\in\mathbb{R}^+$ with $0<\zeta<1$. So $G'$ is $\left( n^{-\zeta},h,1\right) $-typical for every fixed $h\in \mathbb{N}^+$ and $\zeta\in \mathbb{R}^+$ with $0<\zeta<1$. Applying Theorem \ref{Thm:lower_decom} with $\left(C_T^k,G',\zeta,1\right)$ playing the role of $(F,G,\zeta,p)$, we obtain that the number of distinct $C_T^k$-decompositions of $G'$ is at least $n^{\frac{T-k-\beta'}{T}|G'|},$ for arbitrary fixed $\beta'\in \mathbb{R}^+$ with $0<\beta'<1$. Note that $|G'|=|K_n^k|-|W|\ge {n\choose k}-2{n\choose k-1}\frac{\log^3n}{k}$.

    Let $\mathcal{C}$ be a $C_{T}^k$-decomposition of $G'$. For any $C\in \mathcal{C}$, suppose $C=(v_1,v_2,\ldots,v_{T})$. Since $W$ is spanning, one can merge $C$ into $W$ by $W_1v_1\cdots v_{T}v_1\cdots v_{k-1}W_2$, where $W_1v_1\cdots v_{k-1}W_2=W$, which gives a tight closed walk with edge set $E(W)\cup E(C)$. Inserting all cycles of $\mathcal{C}$ in this way yields a tight Euler tour of $K_n^k$. We first claim that different $C_T^k$-decompositions give different tight Euler tours. Indeed, the edges of the fixed walk $W$ are identifiable in the resulting tight Euler tour. Starting immediately after each edge of $W$, the added edges between two consecutive edges of $W$ are uniquely partitioned into consecutive blocks of length $T$, and the edge sets of these blocks are precisely the members of the original $C_T^k$-decomposition. Thus the decomposition can be recovered from the resulting tight Euler tour, and distinct decompositions give distinct tight Euler tours. Hence, the number of distinct tight Euler tours in $K_n^k$ is at least
    \begin{align*}
        n^{\frac{T-k-\beta'}{T}|G'|} \ge & n^{\frac{T-k-2\beta'}{T}{n\choose k}} \ge n^{(1-\beta ){n\choose k}},
    \end{align*}
    since $\frac{k+2}{T}\le \beta$.
\end{proof}
We remark that it is possible to calculate the number of different ways to glue a fixed $C_T^k$-decomposition $\mathcal{C}$ of $G'$ on $W$, but these additional choices are negligible compared with the main term of the construction above. The following theorem states the upper bound for the number of distinct universal cycles for $k$-subsets of $[n]$.

\begin{proof}[Proof of Theorem \ref{Thm:uc}]
    Let $N_{n,k}$ denote the number of distinct universal cycles for $k$-subsets of $[n]$. By Theorem \ref{Thm:uc_lower}, $N_{n,k}>0$ for all sufficiently large $n$. Define $ \gamma(n):=1-\frac{\log_n N_{n,k}}{\binom{n}{k}}.$ We will show $0<\gamma(n)=o_n(1)$. Since a universal cycle for $k$-subsets of $[n]$ corresponds to a tight Euler tour of $K_n^k$, applying Theorem~\ref{prop:euler_tour_upper_bound} with $\ell=k-1$ and $G=K_n^k$, we obtain $N_{n,k}\le k!n^{\binom{n}{k}-1}<n^{\binom{n}{k}}$. Hence $\gamma(n)>0$. On the other hand, by Theorem~\ref{Thm:uc_lower}, for every fixed $\beta>0$ and all sufficiently large $n$, $N_{n,k}\ge n^{(1-\beta)\binom{n}{k}}$, and hence $\gamma(n)\le \beta.$ Since $\beta>0$ is arbitrary and $\gamma(n)>0$, we have $\gamma(n)=o_n(1)$. Therefore, $N_{n,k}=n^{(1-\gamma(n))\binom{n}{k}},$ where $0<\gamma(n)=o_n(1)$.
\end{proof}

\subsection{Counting Euler $\ell$-tours in typical $k$-graphs}
In this subsection, we focus on the proof of Theorem \ref{Thm:lcyclecounting}. The main ingredient is the following lower bound for the enumeration of the number of Euler $\ell$-tours in typical $k$-graphs.

\begin{thm}\label{Thm:lcycleexist}
    Let $k,\ell\in \mathbb{N}^+$ be fixed with $0<\ell <k$ and $\zeta, p\in \mathbb{R}^+$ be fixed with $0<p\le 1$. Then for every fixed $\beta\in\mathbb{R}^+$ with $0< \beta<1$, and all sufficiently large $h=h(\beta)\in\mathbb N^+$, the following holds for all sufficiently large $n\in\mathbb N^+ $. Let $G$ be an $\left(n^{-\zeta},h,p\right)$-typical $k$-graph on $n$ vertices. Suppose $\frac{k}{k-\ell}\mid d_G(v)$ for every $v\in V(G)$ when $(k-\ell)\mid k$, and no additional divisibility condition when $(k-\ell)\nmid k.$ Then
    \begin{itemize}
        \item[$(1)$] the $k$-graph $G$ has an Euler $\ell$-tour;
        \item[$(2)$] the number of distinct Euler $\ell$-tours of $G$ is at least $n^{(1-\beta)(k-\ell)|G|}$.
    \end{itemize}
\end{thm}

\begin{proof}
    Fix $\beta\in \mathbb{R}^+$ with $0<\beta<1$. Let $k,\ell\in \mathbb{N}^+$ be fixed with $0<\ell <k$ and $\zeta, p\in \mathbb{R}^+$ be fixed with $0<p\le 1$. Since an $\left(n^{-\zeta},h,p\right)$-typical $k$-graph is also $\left(n^{-\zeta_0},h,p\right)$-typical for every $0<\zeta_0\le\zeta$, by decreasing $\zeta$ if necessary, we may assume that $0<\zeta<1$. Choose a constant $T\in\mathbb{N}^+$ such that $T\ge \max\left\{2k,\frac{k+2}{\beta}\right\}$, $(k-\ell)\mid T$ and $\gcd\left(k-\ell,\frac{T}{k-\ell}\right)=1$. Let $h\ge h_0:=2^k{2T\cdot T!+k \choose k}$.  
    
    \vspace{0.8 em}
    \noindent\textbf{Construction of a suitable spanning closed $\ell$-walk of $G$. } We have the following claim.
    \begin{claim}\label{claim:lwalk}
        There exists a spanning closed $\ell$-walk $W$ in $G$ such that $\Delta_{k-1}(W)\le 2\log ^3n$, and $|G|-|E(W)| \equiv 0 \pmod{\frac{T}{k-\ell}}.$
    \end{claim}
    \begin{proof}
        Note that $h\ge 2^k{2T\cdot T!+k \choose k}\ge k$. By Lemma \ref{Lemma00}, there exists a spanning tight walk $W_0=v_1v_2\cdots v_{t}$ in $G$ such that $\Delta_{k-1}(W_0)\le \log ^3n$. Choose a constant $t'\in\mathbb{N}^+$ satisfying $t+t'\equiv1\pmod{k-\ell}$, $|G|\equiv t+t'\pmod{\frac{T}{k-\ell}}$ and $k\le t'< k+T$. Such a $t'$ exists by the Chinese Remainder Theorem since $\gcd\left(k-\ell,\frac{T}{k-\ell}\right)=1$. By the same argument as the proof of Lemma \ref{Lemma0}, we can extend the spanning tight walk $W_0$ to a spanning tight closed walk $W_1=v_1v_2\cdots v_{t+t'}$ of $G$ such that $\Delta_{k-1}(W_1)\le 2\log ^3n$. Set $N:=t+t'$. For each $i\in[N]$, let $e_i:=\{v_i,v_{i+1},\ldots,v_{i+k-1}\},$ where subscripts are taken modulo $N$. Note that $E(W_1)=\{e_1,\dots,e_N\}$, so that $e_1,\ldots,e_N\in E(G)$ and are pairwise distinct. Then we define $W$ to consist of the vertex sequence
        \[\underbrace{\left( v_1,v_2,\dots, v_N, v_1,v_2,\dots, v_N, \dots, v_1,v_2,\dots, v_N\right) }_{k-\ell\text{ copies}}\]
        and the edge sequence $\left(e_{\sigma(1)},e_{\sigma(2)},\dots, e_{\sigma(N)}\right),$ where $\sigma(j)$ is the unique element of $[N]$ satisfying $\sigma(j)\equiv 1+(j-1)(k-\ell)\pmod N,$ for each $j\in[N]$. Note that $\sigma(1),\ldots,\sigma(N)$ are pairwise distinct since $N\equiv1\pmod{k-\ell}$. Indeed, if $\sigma(i)=\sigma(j)$ for some $i,j\in[N]$, then $(i-j)(k-\ell)\equiv0\pmod N,$ which implies $i\equiv j\pmod N$, and hence $i=j$. Thus, $\sigma$ is a permutation of $[N]$. Moreover, since $e_1,\ldots,e_N$ are pairwise distinct edges in $G$, $e_{\sigma(1)},\ldots,e_{\sigma(N)}$ are also pairwise distinct edges in $G$. Since $\left\{v_{1+(j-1)(k-\ell)},\ldots, v_{k+(j-1)(k-\ell)}\right\}=e_{\sigma(j)},$ for each $j\in[N]$, where the subscripts on the vertices are taken modulo $N$, we have that $W$ is a closed $\ell$-walk with $E\left(W\right)=E\left(W_1\right)$. It follows that $\Delta_{k-1}(W)=\Delta_{k-1}(W_1)\le 2\log ^3n$.

        It remains to show that $W$ is spanning. Let $D$ be any ordered $\ell$-tuple of distinct vertices of $G$. Extend $D$ to an ordered $(k-1)$-tuple $D'$ of distinct vertices of $G$ such that $D$ consists of the first $\ell$ elements of $D'$. Since $W_1$ is spanning, there exists some $s\in[N]$ such that $D'=(v_s,v_{s+1},\ldots,v_{s+k-2})$, where subscripts are taken modulo $N$. In particular, $D=(v_s,v_{s+1},\ldots,v_{s+\ell-1}).$ Since consecutive terms of the edge sequence $\left(e_{\sigma(1)},\ldots,e_{\sigma(N)}\right)$ have indices differing by $k-\ell$ modulo $N$, the edges $e_{s-(k-\ell)}$ and $e_s$ occur consecutively in $W$. Their prescribed ordered overlap is $(v_s,v_{s+1},\ldots,v_{s+\ell-1})=D.$ Therefore, $W$ is spanning.

        Hence, $W$ is a spanning closed $\ell$-walk in $G$ with $\Delta_{k-1}(W)\le 2\log ^3n$. Since $|E(W)|=N=t+t',$ and $|G|\equiv t+t'\pmod{\frac{T}{k-\ell}}$, we have $|G|-|E(W)| \equiv0\pmod {\frac{T}{k-\ell}}.$ This proves the claim.
    \end{proof}
    
    \vspace{0.8 em}
    \noindent\textbf{Construction of an Euler $\ell$-tour of $G$.}
    Set $G':=G\setminus W$. Recall $h\ge h_0:=2^k{2T\cdot T!+k \choose k}$. Since $\Delta_{k-1}(W)\le 2\log ^3n$, $G'$ is $\left(n^{-\frac{\zeta}{2}},h_0,p\right)$-typical. Since $|G'|=|G|-|E(W)|$ and $\left|C_T^{k,\ell}\right|=\frac{T}{k-\ell}$, Claim \ref{claim:lwalk} yields $\left|C_T^{k,\ell}\right|\mid |G'| .$ Moreover, if $k-\ell\mid k$, then every vertex of $C_T^{k,\ell}$ has degree $\frac{k}{k-\ell}$. Since $\frac{k}{k-\ell}\mid d_G(v)$ for every $v\in V(G)$ and $\frac{k}{k-\ell}\mid d_W(v)$ for every $v\in V(W)$, we have $\frac{k}{k-\ell }\mid d_{G'}(v)$ for every $v\in V(G')$. If $(k-\ell)\nmid k$, then the vertex degrees in $C_T^{k,\ell}$ are $\left\lfloor k/(k-\ell)\right\rfloor$ and $\left\lceil k/(k-\ell)\right\rceil$. Since these are consecutive integers, their greatest common divisor is $1$, and hence $d_{C_T^{k,\ell}}(1)=1$. Furthermore, since $T\ge 2k$, for every $2\le j\le k-1$, the degrees of the $j$-sets in $C_T^{k,\ell}$ have greatest common divisor one. Indeed, for every $2\le j\le k-1$, one may choose a $j$-set containing the two end vertices of an edge of $C_T^{k,\ell}$ and otherwise contained in that edge; since $T\ge 2k$, this set lies in exactly one edge, and hence $d_{C_T^{k,\ell}}(j)=1$. Therefore, $G'$ is $C_T^{k,\ell}$-divisible. Choose a constant $c>0$ such that $c\le \frac{0.9\left(\frac{p}{2}\right)^{h_0}} {(2T\cdot T!)^k4^{2T\cdot T!}}.$ Since $G'$ is $\left(n^{-\zeta/2},h_0,p\right)$-typical and $n^{-\zeta/2}=o(1)$, $G'$ is also $(c,h_0,p)$-typical. By Lemma \ref{Lem:glock_existence}, $G'$ has a $C_T^{k,\ell}$-decomposition $\mathcal{C}$. For each $C\in \mathcal{C}$, suppose $C=(v_1,v_2,\ldots,v_{T})$. Since $W$ is spanning, one can merge $C$ into $W$ by $W_1v_1\cdots v_{T}v_1\cdots v_{\ell}W_2$, where $W_1v_1\cdots v_{\ell}W_2=W$ and $v_1\cdots v_{\ell}$ is the intersection of two consecutive edges of $W$, which gives a closed $\ell$-walk with edge set $E(W)\cup E(C)$. Inserting all cycles of $\mathcal{C}$ in this way yields an Euler $\ell$-tour of $G$.

    \vspace{0.8 em}
    \noindent\textbf{Counting Euler $\ell$-tours via $C_T^{k,\ell}$-decompositions.}
    Note that different $C_T^{k,\ell}$-decompositions of $G'$ give different Euler $\ell$-tours of $G$. Indeed, the edges of the fixed walk $W$ are identifiable in the resulting Euler $\ell$-tour. Starting immediately after each edge of $W$, the added edges between two consecutive edges of $W$ are uniquely partitioned into consecutive blocks of length $\frac{T}{k-\ell}$, and the edge sets of these blocks are precisely the members of the original $C_T^{k,\ell}$-decomposition. Thus the decomposition can be recovered from the resulting Euler $\ell$-tour, and distinct decompositions give distinct Euler $\ell$-tours. By Theorem \ref{Thm:lower_decom}, the number of distinct $C_T^{k,\ell}$-decompositions of $G'$ is at least $n^{\frac{(T-k-\beta')(k-\ell)}{T}|G'|}$ for arbitrary fixed $\beta'\in \mathbb{R}^+$ with $0<\beta'<1$. Since
    \[|G'|=|G|-|W|\ge (1-n^{-\zeta})pn{n\choose k-1}\frac{1}{k}-2{n\choose k-1}\frac{\log^3n}{k},\]
    we get that the number of distinct Euler $\ell$-tours in $G$ is at least
    \[n^{\frac{(T-k-\beta')(k-\ell)}{T}|G'|}\ge n^{\frac{T-k-2\beta'}{T}(k-\ell)|G|}\ge n^{(1-\beta)(k-\ell)|G|}\]
    since $\frac{k+2}{T}\le \beta$.
\end{proof}

\begin{proof}[Proof of Theorem \ref{Thm:lcyclecounting}]
    Let $N_\ell(G)$ denote the number of distinct Euler $\ell$-tours of $G$. By Theorem \ref{Thm:lcycleexist} (1), $N_\ell(G)>0$. Define $\gamma(n):=k-\ell-\frac{\log_n N_\ell(G)}{|G|}.$ By Theorem \ref{prop:euler_tour_upper_bound}, $N_\ell(G) \le k!n^{(k-\ell)(|G|-1)} <n^{(k-\ell)|G|}$, and hence $\gamma(n)>0$. On the other hand, by Theorem \ref{Thm:lcycleexist} (2), for every fixed $\beta>0$, there exists $h_0=h_0(\beta)$ such that for all $h\ge h_0$, $N_\ell(G)\ge n^{(1-\beta)(k-\ell)|G|},$ which implies $\gamma(n)\le \beta(k-\ell).$ Since $h=h(n)\ge h_0$ for every fixed $\beta$, we have $\gamma(n)=o_n(1)$. Therefore, $N_\ell(G)=n^{(k-\ell-\gamma(n))|G|}$ for some $0<\gamma(n)=o_n(1)$, as required.
\end{proof}

\section{Conclusion}\label{Sec:Conclusion}
In this work, we primarily estimate the asymptotic number of distinct $F$-decompositions of a typical $k$-graph $G$. We remark that the present framework should also extend, after some additional bookkeeping, to the enumeration of $(F,\lambda)$-designs of a typical $k$-graph $G$. Nevertheless, we do not pursue this extension here, as it would not add substantial insight and we prefer to keep the exposition simple. Furthermore, it would be interesting to obtain a sharper asymptotic estimate by determining a more precise error term in Theorem~\ref{Thm:count_decom}. Another worthwhile direction is to consider the asymptotic count of the $F$-decompositions of a Dirac-type graph $G$. Moreover, it would be interesting to consider the corresponding problem when replacing the $k$-graph $F$ by a family of $k$-graphs $\mathcal{F}$.

As an application, we also give an asymptotic count of the number of universal cycles for $k$-subsets of $[n]$ and of the number of Euler $\ell$-tours in a typical $k$-graph $G$. In this direction, it would be of great interest to determine the exact count, though this remains an extremely challenging open problem. Furthermore, investigating the enumeration of universal cycles for other objects, such as multisets, partitions and subspaces, constitutes an important direction for future research.

\section*{Acknowledgments}
The authors would like to thank Jie Han for bringing this problem
to their attention and for helpful discussions and valuable comments.

\medskip
\noindent\textbf{Funding.}
Bin Wang was supported by the China Postdoctoral Science Foundation
(Grant No. 2026M794610).
Junxue Zhang was supported by the National Natural Science Foundation
of China (Grant No. 12601655).

\bigskip
\noindent\textbf{Competing interests.}
The authors declare no competing interests.

\bigskip
\noindent\textbf{Data availability.}
Data availability is not applicable to this article
as no datasets were generated or analysed
during the study.

\bigskip
\noindent
\textbf{Declaration of the use of generative AI.}
Generative AI (ChatGPT) was only used to assist with  language polishing. All proof ideas are entirely due to the authors, who take full responsibility for the accuracy and content of this work.
The manuscript was prepared by the authors.

\bibliographystyle{plain}
\bibliography{ref}

@article{keevash2026terminalleaverandomgreedy,
      title={The terminal leave in random greedy hypergraph matching}, 
      author={Peter Keevash},
      year={2026},
      eprint={2610.20414},
      archivePrefix={arXiv},
      primaryClass={math.CO},
      url={https://arxiv.org/abs/2610.20414}, 
}

@article{kwan2020almost,
  author        = {Kwan, M.},
  title         = {Almost all {Steiner} triple systems have perfect matchings},
  journal       = {Proc. Lond. Math. Soc. (3)},
  volume        = {121},
  number        = {6},
  pages         = {1468--1495},
  year          = {2020},
  doi           = {10.1112/plms.12373}
}

@article{MR2745697,
  author        = {Lonc, Z. and Naroski, P.},
  title         = {On tours that contain all edges of a hypergraph},
  journal       = {Electron. J. Combin.},
  volume        = {17},
  number        = {1},
  pages         = {R144},
  year          = {2010},
  doi           = {10.37236/416}
}

@article{MR4622724,
  author        = {\v{S}ajna, M. and Wagner, A.},
  title         = {{$\ell$}-covering {$k$}-hypergraphs are quasi-{Eulerian}},
  journal       = {Discuss. Math. Graph Theory},
  volume        = {43},
  number        = {4},
  pages         = {1091--1102},
  year          = {2023},
  doi           = {10.7151/dmgt.2422}
}

@article{MR3691547,
  author        = {Bahmanian, A. and \v{S}ajna, M.},
  title         = {Quasi-{Eulerian} hypergraphs},
  journal       = {Electron. J. Combin.},
  volume        = {24},
  number        = {3},
  pages         = {P3.30},
  year          = {2017},
  doi           = {10.37236/6361}
}

@article{MR3843267,
  author        = {Steimle, Y. and \v{S}ajna, M.},
  title         = {Spanning {Euler} tours and spanning {Euler} families in hypergraphs with particular vertex cuts},
  journal       = {Discrete Math.},
  volume        = {341},
  number        = {10},
  pages         = {2808--2819},
  year          = {2018},
  doi           = {10.1016/j.disc.2018.06.021}
}

@article{MR5052126,
  author        = {Bahmanian, A. and Shan, S.},
  title         = {Spanning {Euler} tours in hypergraphs},
  journal       = {J. Combin. Theory Ser. A},
  volume        = {222},
  pages         = {106194},
  year          = {2026},
  doi           = {10.1016/j.jcta.2026.106194}
}

@inproceedings{MR396347,
  author        = {Wilson, R. M.},
  title         = {Decompositions of complete graphs into subgraphs isomorphic to a given graph},
  booktitle     = {Proc. 5th British Combinatorial Conf.},
  series        = {Congr. Numer.},
  volume        = {15},
  pages         = {647--659},
  publisher     = {Utilitas Math.},
  year          = {1976}
}

@article{MR366695,
  author        = {Wilson, R. M.},
  title         = {An existence theory for pairwise balanced designs. {III}. {P}roof of the existence conjectures},
  journal       = {J. Combin. Theory Ser. A},
  volume        = {18},
  number        = {1},
  pages         = {71--79},
  year          = {1975},
  doi           = {10.1016/0097-3165(75)90067-9}
}

@article{MR304204,
  author        = {Wilson, R. M.},
  title         = {An existence theory for pairwise balanced designs. {II}. {T}he structure of {PBD}-closed sets and the existence conjectures},
  journal       = {J. Combin. Theory Ser. A},
  volume        = {13},
  number        = {2},
  pages         = {246--273},
  year          = {1972},
  doi           = {10.1016/0097-3165(72)90029-5}
}

@article{MR304203,
  author        = {Wilson, R. M.},
  title         = {An existence theory for pairwise balanced designs. {I}. {C}omposition theorems and morphisms},
  journal       = {J. Combin. Theory Ser. A},
  volume        = {13},
  number        = {2},
  pages         = {220--245},
  year          = {1972},
  doi           = {10.1016/0097-3165(72)90028-3}
}

@article{Kirkman1847,
  author        = {Kirkman, T. P.},
  title         = {On a problem in combinations},
  journal       = {Cambridge Dublin Math. J.},
  volume        = {2},
  pages         = {191--204},
  year          = {1847}
}

@article{MR3711683,
  author        = {Lonc, Z. and Naroski, P. and Rz\k{a}\.{z}ewski, P.},
  title         = {Tight {Euler} tours in uniform hypergraphs---computational aspects},
  journal       = {Discrete Math. Theor. Comput. Sci.},
  volume        = {19},
  number        = {3},
  pages         = {Art.~2},
  year          = {2017},
  doi           = {10.23638/DMTCS-19-3-2}
}

@inproceedings{keevash2018hypergraphmatchingsdesigns,
  author        = {Keevash, P.},
  title         = {Hypergraph matchings and designs},
  booktitle     = {Proc. Int. Congr. Math. 2018},
  volume        = {3},
  pages         = {3099--3122},
  publisher     = {World Scientific},
  year          = {2019},
  doi           = {10.1142/9789813272880_0174}
}

@article{keevash2018existencedesignsii,
  author        = {Keevash, P.},
  title         = {The existence of designs {II}},
  year          = {2018},
  note          = {arXiv:1802.05900}
}

@article{keevash2024existencedesigns,
  author        = {Keevash, P.},
  title         = {The existence of designs},
  journal          = {Ann. of Math., to appear; arXiv:1401.3665}
}

@article{delcourt2024proofhighgirthexistence,
  author        = {Delcourt, M. and Postle, L.},
  title         = {Proof of the high girth existence conjecture via refined absorption},
  year          = {2024},
  note          = {arXiv:2402.17856}
}

@article{delcourt2024refinedabsorptionnewproof,
  author        = {Delcourt, M. and Postle, L.},
  title         = {Refined absorption: a new proof of the existence conjecture},
  year          = {2024},
  note          = {arXiv:2402.17855}
}

@inproceedings{postle2025refinedabsorptionnewproof,
  author        = {Postle, L.},
  title         = {Refined absorption: a new proof of the existence conjecture and its applications to extremal and probabilistic design theory},
  booktitle     = {Proc. Int. Congr. Math. 2026},
  volume        = {6},
  pages         = {212--231},
  publisher     = {SIAM},
  year          = {2026},
  doi           = {10.1137/25M1805977}
}

@article{MR1445383,
  author        = {Grable, D. A.},
  title         = {On random greedy triangle packing},
  journal       = {Electron. J. Combin.},
  volume        = {4},
  number        = {1},
  pages         = {R11},
  year          = {1997},
  doi           = {10.37236/1296}
}

@article{MR3350225,
  author        = {Bohman, T. and Frieze, A. and Lubetzky, E.},
  title         = {Random triangle removal},
  journal       = {Adv. Math.},
  volume        = {280},
  pages         = {379--438},
  year          = {2015},
  doi           = {10.1016/j.aim.2015.04.015}
}

@article{MR4015657,
  author        = {Bennett, P. and Bohman, T.},
  title         = {A natural barrier in random greedy hypergraph matching},
  journal       = {Combin. Probab. Comput.},
  volume        = {28},
  number        = {6},
  pages         = {816--825},
  year          = {2019},
  doi           = {10.1017/s0963548319000051}
}

@inproceedings{joos2025hypergraphremovalprocess,
  author        = {Joos, F. and K{\"u}hn, M.},
  title         = {The hypergraph removal process},
  booktitle     = {Proc. 57th {ACM STOC}},
  pages         = {301--309},
  publisher     = {ACM},
  year          = {2025},
  doi           = {10.1145/3717823.3718110}
}

@article{keevash2024short,
  author        = {Keevash, P.},
  title         = {A short proof of the existence of designs},
  year          = {2024},
  note          = {arXiv:2411.18291}
}

@article{1992Universal,
  author        = {Chung, F. and Diaconis, P. and Graham, R.},
  title         = {Universal cycles for combinatorial structures},
  journal       = {Discrete Math.},
  volume        = {110},
  number        = {1--3},
  pages         = {43--59},
  year          = {1992},
  doi           = {10.1016/0012-365X(92)90699-G}
}

@article{Brad2009Research,
  author        = {Jackson, B. and Stevens, B. and Hurlbert, G.},
  title         = {Research problems on {Gray} codes and universal cycles},
  journal       = {Discrete Math.},
  volume        = {309},
  number        = {17},
  pages         = {5341--5348},
  year          = {2009},
  doi           = {10.1016/j.disc.2009.04.002}
}

@article{2022Enumerations,
  author        = {Chang, Z. and Xue, J.},
  title         = {Enumerations of universal cycles for {$k$}-permutations},
  journal       = {Discrete Math.},
  volume        = {345},
  number        = {9},
  pages         = {112975},
  year          = {2022},
  doi           = {10.1016/j.disc.2022.112975}
}

@article{2020Euler,
  author        = {Glock, S. and Joos, F. and K{\"u}hn, D. and Osthus, D.},
  title         = {Euler tours in hypergraphs},
  journal       = {Combinatorica},
  volume        = {40},
  number        = {5},
  pages         = {679--690},
  year          = {2020},
  doi           = {10.1007/s00493-020-4046-8}
}

@article{2015Counting,
  author        = {Keevash, P.},
  title         = {Counting designs},
  journal       = {J. Eur. Math. Soc.},
  volume        = {20},
  number        = {4},
  pages         = {903--927},
  year          = {2018},
  doi           = {10.4171/JEMS/779}
}

@article{MR1226137,
  author        = {Jackson, B. W.},
  title         = {Universal cycles of {$k$}-subsets and {$k$}-permutations},
  journal       = {Discrete Math.},
  volume        = {117},
  number        = {1--3},
  pages         = {141--150},
  year          = {1993},
  doi           = {10.1016/0012-365X(93)90330-V}
}

@article{MR3066357,
  author        = {Rudoy, Y.},
  title         = {An inductive approach to constructing universal cycles on the $k$-subsets of $[n]$},
  journal       = {Electron. J. Combin.},
  volume        = {20},
  number        = {2},
  pages         = {P18},
  year          = {2013},
  doi           = {10.37236/2852}
}

@article{debruijn1946,
  author        = {de Bruijn, N. G.},
  title         = {A combinatorial problem},
  journal       = {Nederl. Akad. Wetensch., Proc.},
  volume        = {49},
  number        = {7},
  pages         = {758--764},
  year          = {1946}
}

@article{2000On,
  author        = {Hurlbert, G.},
  title         = {On universal cycles for {$k$}-subsets of an {$n$}-set},
  journal       = {SIAM J. Discrete Math.},
  volume        = {7},
  number        = {4},
  pages         = {598--604},
  year          = {1994},
  doi           = {10.1137/S0895480191220861}
}

@article{glock2020existencedesignsiterativeabsorption,
  author        = {Glock, S. and K{\"u}hn, D. and Lo, A. and Osthus, D.},
  title         = {The existence of designs via iterative absorption: hypergraph {$F$}-designs for arbitrary {$F$}},
  journal       = {Mem. Amer. Math. Soc.},
  volume        = {284},
  number        = {1406},
  year          = {2023},
  doi           = {10.1090/memo/1406}
}

@article{freedman1975tail,
  author        = {Freedman, D. A.},
  title         = {On tail probabilities for martingales},
  journal       = {Ann. Probab.},
  volume        = {3},
  number        = {1},
  pages         = {100--118},
  year          = {1975},
  doi           = {10.1214/aop/1176996452}
}

\end{document}